\documentclass{amsart}

\usepackage{times,cite}
\usepackage[mathscr]{eucal}
\usepackage{amsmath,amssymb}

\usepackage{multirow}

\newtheorem{thm}{Theorem}[section]

\newtheorem{cor}[thm]{Corollary}
\newtheorem{prop}[thm]{Proposition}

\newtheorem{lemma}[thm]{Lemma}

\theoremstyle{definition}

\newtheorem{remark}[thm]{Remark}
\newtheorem{example}[thm]{Example}

\DeclareMathOperator{\Sym}{Sym}

\DeclareMathOperator{\Pic}{Pic}

\DeclareMathOperator{\Aut}{Aut}
\DeclareMathOperator{\Fix}{Fix}
\DeclareMathOperator{\rank}{rank}
\DeclareMathOperator{\sing}{sing}
\DeclareMathOperator{\Fol}{Fol}
\DeclareMathOperator{\Aff}{Aff}
\DeclareMathOperator{\Rat}{Rat}
\DeclareMathOperator{\Log}{Log}

\DeclareMathOperator{\codim}{codim}
\DeclareMathOperator{\Gr}{Gr}
\DeclareMathOperator{\PBL}{LPB}

\def\C{\mathbb C}
\def\P{\mathbb P}

\def\Z{\mathbb Z}
\def\F{\mathcal F}

\def\G{\mathcal G}

\input xy
\xyoption{all}

\begin{document}

\title[Foliations on Fano $3$-folds]
{Foliations with trivial canonical bundle on Fano $3$-folds}
\author[F. Loray, J.V. Pereira and F. Touzet ]
{Frank LORAY$^1$, Jorge Vit\'{o}rio PEREIRA$^{2}$ and Fr\'ed\'eric TOUZET$^1$}
\address{\newline $1$ IRMAR, Campus de Beaulieu, 35042 Rennes Cedex, France\hfill\break
$2$ IMPA, Estrada Dona Castorina, 110, Horto, Rio de Janeiro,
Brasil} \email{$^1$ frank.loray@univ-rennes1.fr, frederic.touzet@univ-rennes1.fr}


\maketitle

\begin{abstract}
We classify the irreducible components of the space of foliations on Fano $3$-folds with rank one Picard group.
As a corollary we obtain a classification of holomorphic Poisson structures on the same class of $3$-folds.
\end{abstract}
\setcounter{tocdepth}{1}
 \sloppy

\section{Introduction}
Let $X$ be a  projective manifold and $N$ be a line bundle on it. A holomorphic $1$-form with coefficients in $N$ defines a
codimension one foliation $\mathcal F$ if and only if it satisfies the Frobenius integrability condition
$\omega \wedge d \omega = 0$  in $H^0(X, \Omega^3_X \otimes N^{\otimes 2}).$
If this is the case and $\omega$  has zeros of codimension at least two then $N$ is called the normal bundle of $\F$. For a fixed line bundle
$N$ on a fixed projective manifold $X$, it is natural to study the irreducible components of the variety
\[
\textup{Fol}(X,N) = \left\{ [\omega] \in \mathbb P H^0(X, \Omega^1_X \otimes N) \, \, \big\vert \, \,  \omega \wedge d \omega =0 ; \codim \sing(\omega) \ge 2 \right\}
\]
which we call the space of codimension one foliations on $X$ with normal bundle $N$.

If $X$ has dimension two then the integrability condition is automatically satisfied and the space of foliations with a given normal bundle $N$
is either empty or has only one irreducible component which is an open subset of the projective space $\mathbb P H^0(X, \Omega^1_X \otimes N)$.  The discussion from now one will focus on projective manifolds of dimension at least three.
When $X= \mathbb P^n$ the normal bundle of a codimension one foliation $\mathcal F$ is $\mathcal O_{\mathbb P^n}(d+2)$ where $d$ is the degree of the foliation
defined as the number of tangencies $\mathcal F$ with a general line. The irreducible components of $\Fol(\mathbb P^n, d) = \Fol(\mathbb P^n,\mathcal O_{\mathbb P^n}(d+2))$
for $d = 0$ and $d=1$ are described by Jouanolou in \cite{EPA} using elementary methods. In the celebrated work \cite{CerveauLinsNetoAnnals}, Cerveau and Lins Neto give a
complete description of the irreducible components of $\Fol(\mathbb P^n, 2)= \Fol(\mathbb P^n, \mathcal O_{\mathbb P^n}(4))$, $n \ge 3$. The methods are considerably more involved and rely on the study of the Gauss map
of the foliations, Dulac's classification of centers of degree 2 polynomial planar vector fields \cite{Dulac}, and computer-assisted calculations.

The canonical bundle of a foliation $\mathcal F$ with normal bundle $N $ can be defined as $K\mathcal F= K  X\otimes N^*$, where $KX$ is the canonical bundle of the ambient
manifold. Notice that in $\mathbb P^n$ the foliations with trivial canonical bundle are precisely those of degree $n-1$. In the particular case of $\mathbb P^3$, Cerveau-Lins Neto classification  is  the classification of the irreducible components of the space of foliations on $\mathbb P^3$ with trivial canonical bundle.
The main purpose of this paper is to extend this classification to the other Fano $3$-folds ($3$-folds with ample anticanonical bundle) having Picard group isomorphic to $\mathbb Z$.  Our main result and Cerveau-Lins Neto
 classification of foliations of degree two on $\mathbb P^3$ are summarized in Table \ref{Table:1}. For more precise statements see Theorems \ref{T:Quadricas}, \ref{T:CLN}, \ref{T:i2}, and \ref{T:i1}.  In the table $\Rat$, $\Log$, $\PBL$, and $\Aff$ stand for rational, logarithmic, linear pull-back, and affine respectively. Definitions of $\Rat$ and $\Log$ can be found in  Section \ref{S:parametros}; definition of $\PBL$ is in  Section \ref{S:P3}; and the irreducible components that go under the label   $\Aff$ are described in  Example \ref{E:affQ} and in the proofs of Theorems \ref{T:i1} and  \ref{T:i2}. As a corollary we obtain a classification of holomorphic Poisson structures
 on Fano $3$-folds with rank one Picard group, see Section \ref{S:Poisson}.

\begin{center}

\begin{table}[ht]
\begin{tabular}{|c|c|c|c|}
\hline
{\bf Manifold} & {\bf Irreducible component} & {\bf dimension} \\
\hline

\multirow{6}{*}{Projective space $\P^3$}
& $\Rat(1,3)$& $21 $ \\    \cline{2-3}
& $\Rat(2,2)$& $16$ \\    \cline{2-3}
& $\Log(1,1,1,1)$  &  $14$ \\   \cline{2-3}
  &$\Log(1,1,2)$ & $17$\\    \cline{2-3}
  & $\PBL(2)$ & $17$\\    \cline{2-3}
  & $\Aff$ & $13$ \\    \cline{2-3}
\hline\hline
  \multirow{3}{*}{Hyperquadric $Q^3$}
  & $\Rat(1,2)$ & $17$ \\    \cline{2-3}
  & $\Log(1,1,1)$  & $14 $ \\    \cline{2-3}
  & $\Aff$ & $8$ \\    \cline{2-3}
\hline\hline
  Hypersurface of degree $6$
in   $\mathbb P(1,1,1,2,3)$&  $\Rat(1,1)    $ & $2$ \\
\hline  \hline
Hypersurface  of degree $4$
 in $\mathbb P(1,1,1,1,2)$ &  $\Rat(1,1) $ & $ 4$ \\
 \hline  \hline
  Cubic in $\mathbb P^4$ &  $\Rat(1,1)  $ & $6$ \\
  \hline  \hline
  Intersection of quadrics in $\mathbb P^5$  &  $\Rat(1,1)  $ & $8$ \\
  \hline  \hline
  \multirow{2}{*}{$X_5$} &  $\Rat(1,1) $ & $10$ \\
  \cline{2-3}
                  &$ \Aff  $ & $1$
      \\          \hline
\hline
        Mukai-Umemura $3$-fold  &  $\Aff  $ & $1$ \\
        \hline
\end{tabular}
\caption{Irreducible components of the space of foliations with $K\F=0$ on Fano $3$-folds with rank one Picard group.} \label{Table:1}
\end{table}
\end{center}

Our main technical tool is the following result obtained by combining Theorem 3.5 and Theorem 3.8 of  \cite{LPT}.

\begin{thm}\label{T:truquesujo}
Let $(X,H)$ be a polarized complex projective  manifold of dimension $n$ and $\mathcal F$ be a codimension one foliation on $X$ with  numerically trivial
canonical bundle and semi-stable tangent sheaf. Suppose  $c_1(TX)^2 \cdot H^{n-2} >0$.
Then at least one of the following statements holds true:
\begin{enumerate}
\item $T\mathcal F$ is stable and  $\mathcal F$ is a  rationally connected foliation, i.e., the general leaf of $\mathcal F$ is a  rationally connected
algebraic variety;
\item $T\mathcal F$ is strictly semi-stable and there is a rationally  connected  foliation $\mathcal H$ tangent to $\mathcal F$ and with $K\mathcal H \cdot H^{n-1}=0$; or
\item $\mathcal F$ is defined by a closed rational $1$-form with coefficients in a flat line bundle and without divisorial components in its zero set.
\end{enumerate}
\end{thm}

Indeed we will show that when $\Pic(X)=\Z$,  statement (1) implies statement (3). This will be achieved   through a study of fibers of rational maps $F: X \dashrightarrow \mathbb P^1$, which seems
to have some independent interest.

\subsection{Number of reducible fibers of first integrals}
Let $\F$ be a codimension one foliation on a complex projective manifold $X$
defined by the levels of a rational  map $F: X \dashrightarrow C$ from $X$ to some
algebraic curve $C$. If we further assume that $F$ has irreducible general fiber (what can always be done
after replacing $F$ by its Stein factorization) and,   following \cite{Vistoli}, define  its  base number as
\[
r(\mathcal F) = r(F) = \sum_{x \in C} \left( \# \{ \text{ irreducible components of } F^{-1}(x)\} -1 \right) \,  .
\]
then we obtain  a rather strong bound on $r(\F)$ under the additional assumption that  $T\F$ is  stable/semi-stable and has  zero/positive first Chern class.

\begin{thm}\label{TI:A} Let $\F$ be  a codimension one foliation on a polarized complex projective manifold $(X,H)$ of dimension $n \ge 3$.
If $T \mathcal F$ is $H$-semi-stable and $K \F\cdot H^{n-1} < 0$, or   $T \mathcal F$ is $H$-stable and
$K \F \cdot H^{n-1} = 0$,  then
\[
r(\mathcal F) \le \rank NS(X) - 1 \, ,
\]
where $NS(X)$ is the Neron-Severi group of $X$. In particular, if $X = \mathbb P^n$, $n\ge 3$, then $r(\mathcal F)=0$.
\end{thm}

Combining this result with a classical Theorem by Halphen about pencils on projective spaces (which we
generalize to simply connected projective manifolds in  Theorem \ref{T:Halphen}) we are able
to control the first integrals of (semi)-stable foliations on Fano manifolds with rank one Picard group
having (negative) zero canonical bundle.

\subsection{Plan of the paper}  In Section \ref{S:prelim} we have collected basic results about foliations
that will be used in the sequel.
Section \ref{S:Halphen} studies the relationship between
the existence of invariant hypersurfaces and the semi-stability of the tangent sheaf. Besides the  proof of Theorem \ref{TI:A}, it also contains a generalization of a classical result of Halphen, and the classification of  foliations with $K\F <0$ on Fano $3$-folds with rank one Picard group (Proposition \ref{P:verynegative}).
Section \ref{S:classifica}   gives  a rough classification (Theorem \ref{T:rr}) of foliations with trivial canonical bundle on Fano $3$-folds with rank one Picard group.
In Section \ref{S:Q3} we give a complete classification of foliations with $K\F =0 $ on three-dimensional quadrics, Theorem \ref{T:Quadricas}. In Section \ref{S:P3}
 we recall the statement of Cerveau-Lins Neto classification (Theorem \ref{T:CLN}), give a classification of the  foliations on $\mathbb P^n$ of degree one and arbitrary codimension (Theorem \ref{T:deg1}), and show how to deduce
the Cerveau-Lins Neto classification for $n>3$ from the classification for $n=3$ using the classification of foliations of degree one.  Sections \ref{S:i2} and \ref{S:i1} deals
with cases of index two (Theorem \ref{T:i2}) and one (Theorem \ref{T:i1}), respectively. And finally in  Section \ref{S:Poisson} we spell out the classification of holomorphic Poisson structures
on Fano $3$-folds with rank one Picard group in Theorem \ref{T:Poisson}.

\section{Basic concepts}\label{S:prelim}

\subsection{Foliations as subsheaves of the tangent and cotangent bundles}
A foliation $\mathcal F$ on a complex manifold is determined by a coherent subsheaf $T\mathcal F$
of the tangent sheaf $TX$ of $X$ which
\begin{enumerate}
\item is  closed under the Lie bracket (involutive), and
\item the inclusion $T\mathcal F \to TX$ has torsion free cokernel.
\end{enumerate}
The locus of points where $TX / T\mathcal F$ is not locally free is
called the singular locus of $\mathcal F$, denoted here by $\sing(\mathcal F)$. Condition (2) implies, in particular, that the codimension of $\sing(\mathcal F)$ is at least two.
The dimension of $\mathcal F$, $\dim \F$ for short,  is by definition the generic rank of $T \mathcal F$. The codimension of
$\F$, $\codim \F$, is defined as the integer $\dim X - \dim \F$.

The dual of $T\mathcal F$ is the cotangent sheaf of $\mathcal F$ and will be denoted by $T^*\mathcal F$.
The determinant of $T^*\mathcal F$, i.e. $(\wedge^{\dim \F} T^* \mathcal F)^{**}$ is what we will call the canonical bundle of $\mathcal F$ and will be denoted by $K \mathcal F$.

There is a dual point of view where $\mathcal F$ is determined by a subsheaf $N^* \mathcal F$ of the cotangent sheaf $\Omega^1_X = T^* X$ of $X$. The involutiveness
asked for in condition (1) above is replace by integrability: $dN^*\mathcal F \subset N^* \mathcal F \wedge \Omega^1_X$ where $d$ is the exterior derivative.
Condition (2) is unchanged: $\Omega^1_X / N^* \mathcal F$ is torsion free. The normal bundle of $\mathcal F$ is defined as the dual of $N^* \mathcal F$. Over the smooth locus $X - \sing(\mathcal F)$ we have the following exact sequence
\[
0 \to T\mathcal F \to TX \to N \mathcal F \to 0 \, ,
\]
but  this is not  valid over the singular locus. Anyway, as the singular set has codimension at least two
we obtain the adjunction formula
\[
KX = K\mathcal F \otimes \det N^* \mathcal F
\]
valid in the Picard group of $X$.

\subsection{Foliations as   $q$-forms and spaces of foliations}
If $\mathcal F$ is a codimension $q$ foliation on a complex variety $X$ then the $q$-th wedge product of the inclusion
\[
N^* \mathcal F \longrightarrow \Omega^1_X
\]
determines a differential $q$-form  $\omega$ with
coefficients in the line bundle  $\det N \mathcal F = (\wedge^{q} N \mathcal F)^{**}$ having the following properties:
\begin{itemize}
\item {\bf Local decomposability}: the germ of $\omega$ at the general point of $X$ decomposes as the product of $q$ germs of holomorphic $1$-forms
\[
\omega = \omega_1 \wedge \cdots \wedge \omega_q .
\]
\item {\bf Integrability}: the decomposition of $\omega$ at the general point of $X$ satisfies Frobenius integrability condition
\[
d \omega_i \wedge \omega =0 \quad \text{ for every } i = 1 , \ldots, q \, .
\]
\end{itemize}
The tangent bundle of $\mathcal F$ can be recovered as the kernel of the  morphism
\[
TX \to \Omega_X^{q-1}\otimes \det N\mathcal F
\]
defined by contraction with $\omega$.

Reciprocally, if $\omega \in H^0(X, \Omega^q \otimes N)$ is a twisted $q$-form with coefficients in a  line bundle $N$
which is locally decomposable and  integrable then the kernel of $\omega$ has generic rank $\dim X - q$, and it is  the tangent bundle of a
holomorphic foliation $\mathcal F$. Moreover, if the zero set of $\omega$ has codimension at least two then $N = \det N \mathcal F$.

\medskip

\begin{example}[Foliations on $\mathbb P^n$ and homogeneous forms]\label{E:Pn} Let $\mathcal F$ be a codimension $q$-foliation on $\mathbb P^n$ given
by $\omega \in H^0(\mathbb P^n, \Omega^q_{\mathbb P^n} \otimes N)$. If $i: \mathbb P^q \to \mathbb P^n$ is a general linear immersion then
$i^* \omega \in H^0(\mathbb P^q, \Omega^q_{\mathbb P^q} \otimes N)$ is a section of a line bundle, and its
zero divisor reflects the tangencies between $\mathcal F$ and $i(\mathbb P^q)$. The degree of $\mathcal F$ is, by definition, the degree of
such tangency divisor. It is commonly denoted by $\deg(\mathcal F)$.
Since $\Omega^q_{\mathbb P^q}\otimes N= \mathcal O_{\mathbb P^q}( \deg(N) - q - 1)$, it follows that $N= \mathcal O_{\mathbb P^n}(\deg(\mathcal F) + q + 1)$.

The Euler sequence implies that a section $\omega$ of $\Omega^q_{\mathbb P^n} ( \deg(\mathcal F) + q + 1  )$ can be thought as a polynomial $q$-form
on $\C^{n+1}$ with  homogeneous coefficients of degree $\deg(\mathcal F) + 1$, which we will still denote by $\omega$, satisfying  (*) $i_R  \omega = 0$ where
$ R = x_0 \frac{\partial}{\partial x_0} + \cdots + x_n \frac{\partial}{\partial x_n}$ is the radial vector field. Thus the study of foliations of degree
$d$ on $\mathbb P^n$ reduces to the study of locally decomposable, integrable homogeneous $q$-forms of degree $d+1$ on $\mathbb C^{n+1}$ satisfying the relation (*).
\end{example}

\subsection{Harder-Narasimhan filtration}Let $\mathcal E$ be a torsion free coherent sheaf on a  $n$-dimensional smooth projective  variety $X$ polarized
by an ample line bundle $H$. The slope of $\mathcal E$ (more precisely the $H$-slope of $\mathcal E$)  is defined as the quotient
\[
\mu ( \mathcal E) = \frac{ c_1( \mathcal E) \cdot H^{n-1} }{ \rank ( \mathcal E) } \, .
\]
If the slope of every nonzero proper  subsheaf $\mathcal E'$ of $\mathcal E$ satisfies $\mu ( \mathcal E') < \mu ( \mathcal E )$ (respectively $\mu ( \mathcal E') \le \mu ( \mathcal E )$) then $\mathcal E$ is called stable (respectively semi-stable). A sheaf which is semi-stable but not stable is said to be strictly semi-stable.

There exists a unique filtration of $\mathcal E$ by torsion free subsheaves
\[
0 = \mathcal E_0 \subset \mathcal E_1 \subset \cdots \subset \mathcal E_r = \mathcal E
\]
such that  $\mathcal G_i := \mathcal E_i / \mathcal E_{i-1}$ is semi-stable, and
$\mu(\mathcal G_1) > \mu(\mathcal G_2) > \ldots > \mu(\mathcal G_r) \, .$ This filtration is called the Harder--Narasimhan filtration of $\mathcal E$.
Of course $\mathcal E$ is semi-stable if and only if $r=1$.
Usually one writes $\mu_{max}(\mathcal E) = \mu(\mathcal G_1)$ and $\mu_{min}(\mathcal E) = \mu(\mathcal G_r)$.

We will say that a foliation $\mathcal F$ is stable/semi-stable/strictly semi-stable when its tangent sheaf $T\mathcal F$ is stable/semi-stable/strictly semi-stable .
When $\mathcal E$ is the tangent sheaf of a foliation $\mathcal F$,  the proof
of  \cite[Chapter 9, Lemma 9.1.3.1]{kol} (see also \cite[Proposition 2.1]{LPT}) implies the following result.

\begin{prop}\label{P:unstable}
Let $\mathcal F$ be a foliation on a polarized smooth projective variety $(X,H)$ satisfying  $\mu ( T\mathcal F) \ge 0$.
If $\mathcal F$ is not semi-stable then the maximal destabilizing subsheaf of $T \mathcal F$
is involutive. Thus there exists a semi-stable foliation $\mathcal G$ tangent to  $\mathcal F$ and
satisfying
$
\mu ( T \mathcal G) >    \mu ( T \mathcal F ) \, .
$
\end{prop}

\begin{example}\label{E:concreto}
If $\mathcal F$ is a foliation of $\mathbb P^n$ then the slope of $T\mathcal F$ is
\[
\mu(T \mathcal F) = \frac{\dim(\mathcal F) - \deg(\mathcal F) }{\dim(\mathcal F)} \, .
\]
Therefore $T\mathcal F$ is semi-stable if and only if for every distribution $\mathcal D$ tangent to $\mathcal F$ we have
$
\frac{\deg(\mathcal D)}{\dim(\mathcal D)} \ge   \frac{\deg(\mathcal F)}{\dim(\mathcal F)}  \, .
$
Of course, $T \mathcal F$ is stable if and only if the strict inequality holds for every proper distribution $\mathcal D$.

If $ \mathcal F$ is unstable and $\deg(\mathcal F ) \le \dim(\mathcal F)$ then there exists a foliation $\mathcal G$ contained in $\mathcal F$ satisfying
\[
\frac{\deg(\mathcal G) } { \dim (\mathcal G) }  <\frac{\deg(\mathcal F) } { \dim (\mathcal F) }
\, .
\]
\end{example}

\subsection{Miyaoka-Bogomolov-McQuillan Theorem}

We recall that an algebraic  variety $Y$ is  rationally connected  if through any two points $x,y \in Y$ there exists a rational curve $C$ in $Y$ containing $x$ and $y$. Foliations
with all leaves algebraic and with rationally connected general leaf will be called rationally connected foliations. Beware that there exists  rationally connected foliation
with some leaves non rationally connected, see for instance \cite[\S 2.3]{LPT}. A fundamental result in the study of holomorphic foliations is  Miyaoka's Theorem
 see  \cite[Theorem 8.5]{Miyaoka}, \cite[Chapter 9]{kol}, which was later generalized by Bogomolov and McQuillan \cite{BogMac}, \cite{KST}.
Below we present a variation of it which is deduced in \cite[Corollary 2.3]{LPT} from the main statement of \cite{KST}.

\begin{thm}\label{T:MBM}
Let $\mathcal F$ be a semi-stable foliation on a $n$-dimensional polarized projective variety $(X,H)$. If $K\mathcal F \cdot H^{n-1} < 0$ then
$\mathcal F$ is a rationally connected foliation.
\end{thm}

\subsection{Closed $1$-forms without divisorial components in theirs zero sets}\label{S:parametros}

Let $X$ be a simply-connected projective manifold of dimension at least three. If $D  = \sum \lambda_i H_i$ is a $\mathbb C$-divisor on $X$ with zero (complex)  first Chern class
then there exists a unique closed rational $1$-form $\eta= \eta_D$ on $X$ with simple poles and  residue equal to $\sum \lambda_i H_i$. The associated
foliation has normal bundle equal to $\mathcal O_X( (\eta)_{\infty} - (\eta)_0)$. If we multiply $\eta$ by the defining equations of the (reduced) hypersurfaces $H_i$
then we obtain a section of $H^0(X, \Omega^1_X \otimes \mathcal O_X( \sum H_i) )$. If  $X$ has Picard group $\mathbb Z$ (generated by $\mathcal O_X(1)$)
and the divisors $H_i$ have degree $d_i$ then we get a rational
\begin{align*}
 \Phi : \Sigma  \times \left( \prod_{i=1}^k \mathbb P H^0(X, \mathcal O_X(d_i)) \right) &  \dashrightarrow \mathbb P H^0\big(X,\Omega^1_X ( \sum d_i) \big) \\
 \left(  (\lambda_1: \cdots : \lambda_k), f_1, \ldots, f_k \right) &\mapsto \left( \prod_{i=1}^k f_i \right) \left( \sum_{i=1}^k \lambda_i \frac{df_i}{f_i} \right) \, ,
\end{align*}
where $\Sigma \subset \mathbb P^{k-1}$ is the hyperplane $\{\sum \lambda_i d_i = 0\}$. If the domain is not empty then the closure of the image of $\Phi$  will be denoted by $\Log(d_1, \ldots, d_k)$ when $k\ge3$, and  by $\Rat(d_1,d_2)$ when $k=2$.
Under mild assumptions Calvo-Andrade \cite{Omegar} proved that these subvarieties are irreducible components of the space of foliations on $X$ with normal
bundle $\mathcal O_X(\sum d_i)$. They are usually called the logarithmic components (when $k\ge 3$) or the rational components (when $k=2$) of
the space of foliations. Notice that in general the image of  $\Phi$ is not closed:  the confluence of hypersurfaces may  give rise to
indeterminacy points of $\Phi$, and at the closure of its image one may find foliations defined by closed rational $1$-forms with poles of higher order.
Nevertheless, it can be verified that  for  any $1$-form $[\omega] \in \mathbb P H^0(X, \Omega^1_X(\sum d_i))$  in the closure of the image of $\Phi$ there exists
a section $f$  of $H^0( X, \mathcal O_X(\sum d_i))$  such that $f^{-1}\omega$ is a closed rational $1$-form.

Even if the assumptions of Calvo-Andrade's result do  not hold,  a straightforward adaptation of the proof of \cite[Lemma 8]{CerveauLinsNetoAnnals} shows
that any foliation with normal bundle $\mathcal O_X(d)$ defined by a closed rational $1$-form without divisorial components in its zero set lies in a variety of the form $\Log(d_1, \ldots, d_k)$ with $\sum d_i = d$.

\begin{lemma}\label{L:CLN}
Let $X$ be a projective manifold with $H^0(X, \Omega^1_X)=0$. Let  $\mathcal F$ be
codimension one foliation on $X$ defined by
a closed rational $1$-form $\omega$ with zero set of codimension at least two and
polar divisor $(\omega)_{\infty}= \sum_{i=1}^k r_i D_i$. Then there exists a
holomorphic family of foliations $\mathcal F_t$, $t \in (\mathbb C,0)$, such that
\begin{enumerate}
\item $\mathcal F_0 = \mathcal F$;
\item $N\mathcal F_t = N \mathcal F = \mathcal O_X(\sum_{i=1}^k r_i D_i)$ for   every $t\in \mathbb C$; and
\item $\mathcal F_t$ is defined by a logarithmic $1$-form for every $t\neq 0$.
\end{enumerate}
\end{lemma}

\section{First integrals of (semi)-stable  foliations}\label{S:Halphen}

Theorem \ref{T:MBM} tells us that  semi-stable foliations with  negative canonical bundle
have algebraic  leaves and that the general one is rationally connected.  The goal
of this section is to complement this result for codimension one foliations by giving more information about the first integral. We also deal with stable foliations with numerically trivial canonical bundle having rational first integrals, and the  results
here presented will play an important role in proof of the classification of codimension one foliations with $K \F=0$ on Fano $3$-folds with
rank one Picard group.

\subsection{Invariant hypersurfaces and subfoliations}

Let $\mathcal F$ be a foliation of codimension $q$ on a compact K\"{a}hler manifold $X$.
Let $\mathrm{Div} (\mathcal F) \subset \mathrm{Div} (X)$ be the subgroup of the group of  divisors of $X$ generated by irreducible hypersurfaces invariant by $\mathcal F$.
The arguments used in \cite{ghys} to prove Jouanolou's theorem  lead us to the following result.

\begin{lemma}\label{L:invariante}
Suppose the dimension of $\mathcal F$ is greater than or equal to  two.
If $D \in \mathrm{Div}(\mathcal F)$ satisfies
$
c_1(D) = m \cdot c_1( N \mathcal F) \in H^2(X,\mathbb Z)
$
for a suitable $m \in \mathbb Z$ then at least one of the following assertions holds true:
\begin{enumerate}
\item[(a)] the integer $m$ is non-zero  and  $\mathcal F$ is, after a ramified abelian covering of degree $m$ and a
bimeromorphic morphism,  defined by a meromorphic closed $q$-form with coefficients in a flat line bundle; or
\item[(b)] the integer $m$ is zero and  $\mathcal F$ is tangent to a codimension one logarithmic foliation
with poles at the support of $D$ and integral residues; or
\item[(c)] there exists a foliation $\mathcal G$ of codimension $q+1$  tangent  to $\mathcal F$ with normal sheaf
satisfying $$\det N \mathcal G =  \det N \mathcal F \otimes \mathcal O_X(- \Delta)$$ for some
effective divisor $\Delta$.
\end{enumerate}
\end{lemma}
\begin{proof}
Let $N = \det N \mathcal F$ and $\omega \in H^0(X, \Omega^q_X \otimes N)$ be a twisted $q$-form defining $\mathcal F$.
Write $D$ as $\sum \lambda_{\alpha} H_{\alpha}$ with $\lambda_{\alpha} \in \mathbb Z$.

Our hypothesis ensure the existence of an open covering of $\mathcal U= \{ U_i \}$ where
\[
H_{\alpha} \cap U_i = \{ h_{\alpha}^{(i)} = 0 \} \quad   \text{ and }  \quad \sum \lambda_{\alpha} \left( \frac{d h_{\alpha}^{(i)}}{h_{\alpha}^{(i)}} - \frac{d h_{\alpha}^{(j)}}{h_{\alpha}^{(j)}} \right) = m \frac{dg_{ij}}{g_{ij}}
\]
where $\{ g_{ij} \} \in H^1(\mathcal U, \mathcal O_X^*)$ is a cocycle defining $N$, i.e. $\omega$ is defined by a collection of $q$-forms $\{ \omega_i \in \Omega^q_X(U_i) \}$ which satisfies $\omega_i = g_{ij} \omega_j$.

On $U_i$, set $\eta_i = \sum \lambda_{\alpha}  \frac{d h_{\alpha}^{(i)}}{h_{\alpha}^{(i)}}$ and define
\[
\theta_i = \eta_i \wedge \omega_i + m \cdot d \omega_i \, .
\]
As the hypersurfaces $H_{\alpha}$ are invariant by $\mathcal F$, $\theta_i$ is a holomorphic $(q+1)$-form.
It is also clear that  $\theta_i$ is   locally decomposable and integrable. Moreover, on $U_i \cap U_j$ we have the identity
\[
\theta_i = \left( \eta_j - m\frac{dg_{ij}}{g_{ij}} \right) \wedge g_{ij} \omega_j + m \cdot d (g_{ij} \omega_j) = g_{ij} \theta_j \, .
\]
Hence the collection $\{\theta_i\}$ defines a holomorphic section  $\theta$ of $\Omega^{q+1}_X \otimes N$. If this section is nonzero then
it defines a foliation $\mathcal G$ with $\det N \mathcal G = \det N \mathcal F \otimes \mathcal O_X( - (\theta)_0 )$, where $(\theta)_0$ is the divisorial part of the zero scheme of $\theta$. We are in case (c).

Suppose now that $\theta$ is identically zero. If $m=0$ then $\eta_i = \eta_j$ on $U_i \cap U_j$ and we can patch then together to
obtain a logarithmic $1$-form $\eta$ with poles at the support of $D$. Clearly we are in case $(b)$.

If $m\neq 0$ then on $U_i$
the (multi-valued) meromorphic $q$-form
\[
\Theta_i = \exp \left( \int \frac{1}{m} \eta_i \right) \omega_i = \left(\prod h_{\alpha_i}^{\lambda_{\alpha}/m} \right) \omega_i
\]
is closed. Moreover, if $U_i \cap U_j \neq \emptyset$ then  $\Theta_i  = \mu_{ij}\Theta_j$ for suitable $\mu_{ij} \in \mathbb C^*$. It is a simple matter  to
see that we are in case (a).
\end{proof}

\subsection{Number of reducible fibers of first integrals}

Let $\mathcal F$ be a codimension one  foliation  on a polarized projective manifold $(X,H)$ having a rational first integral.
Stein's factorization
ensures the existence of a rational first integral $F: X \dashrightarrow C$ with irreducible general fiber. We
are interested in bounding the number of non-irreducible fibers of $f$. More precisely we want to bound the number
\[
r(\mathcal F) = r(F) = \sum_{x \in C} \left( \# \{ \text{ irreducible components of } F^{-1}(x)\} -1 \right) \,  ,
\]
where we  do not count the multiplicity of the irreducible components of $F^{-1}(x)$.
 This problem, for rational functions $F: X \dashrightarrow \mathbb P^1$ has been investigated by A. Vistoli and others. In \cite{Vistoli}
he obtains a bound in function of the rank of the Neron-Severi group of $X$ and what he calls the base number of $F$. In particular,
when $X$ is $\mathbb P^n$, he proves that $r(F) \le \deg(F)^2 -1$ where $\deg(F)$ is the degree of a general fiber of $F$.
Our result below gives much stronger bounds for the first integrals obtained through  Theorem \ref{T:MBM} when $\dim X \ge 3$.

\begin{thm}\label{T:Vistoli} Suppose the dimension $X$ is  at least three.
If $\mathcal F$ is semi-stable and $c_1(T\mathcal F) \cdot H^{n-1} > 0$, or   $\mathcal F$ is stable and
$c_1(T\mathcal F) \cdot H^{n-1} = 0$  then
\[
r(\mathcal F) \le \rank NS(X) - 1 \, ,
\]
where $NS(X)$ is the Neron-Severi group of $X$. In particular, if $X = \mathbb P^n$, $n\ge 3$, then $r(\mathcal F)=0$.
\end{thm}
\begin{proof}
Let $x_1, \ldots, x_k$ be  the points of $C$ for which $F^{-1}(x)$ is non-irreducible, and let $n_1, \ldots, n_k$
be the number of irreducible components of $F^{-1}(x_i)$. Choose $n_i-1$ irreducible components in each of the
non-irreducible fibers  and denote them by $F_1, \ldots, F_{r(\mathcal F)}$. If $r(\mathcal F) \ge \rank NS(X)$
then an irreducible fiber $F_0$ is numerically equivalent to a $\mathbb Q$-divisor supported on
$F_1 \cup \cdots \cup F_{r(\mathcal F)}$. Thus there exists a nonzero $D \in \mathrm{Div}(\F)$
with zero Chern class and supported on $F_0 \cup \cdots \cup F_{r(\mathcal F)}$.

 Lemma \ref{L:invariante} implies that
either  there exists a codimension two foliation $\mathcal G$ contained in $\mathcal F$ with $\det N \mathcal G = N \mathcal F \otimes \mathcal O_X( - \Delta)$, for some $\Delta \ge 0$; or $\mathcal F$ is defined by a logarithmic $1$-form  $\eta$ with poles in $D$.
We will now analyze these two possibilities.

If there exists $\mathcal G$ as above and $c_1(T\mathcal F) \cdot H^{n-1} <0$ then
\[
0 <
 c_1(T\mathcal F)\cdot H^{n-1} = ( c_1(T \mathcal G) -  \Delta ) \cdot H^{n-1}  \le  c_1(T \mathcal G)\cdot H^{n-1}  \, ,
\]
which implies $\mu(T\mathcal G) > \mu(T\mathcal F)$ contradicting the semi-stability of $\F$. Similarly, when $c_1(T\mathcal F)\cdot H^{n-1}=0$ we deduce $\mu(T\mathcal G) \ge \mu(T\mathcal F)=0$ contradicting the stability of $\F$.

Suppose now that $\mathcal F$ is defined by $\eta$. As the general  fiber of $F$ is irreducible, there exists a $1$-form
$\eta'$ on $C$ such that $\eta = F^* \eta'$. Consequently the polar set of $\eta$ is set-theoretically equal to a  union
of fibers of $F$. This contradicts the choice of  $F_1, \ldots, F_{r(\F)}$, and concludes the proof.
\end{proof}

\subsection{Multiple fibers of rational maps to $\mathbb P^1$} A classical result of Halphen \cite[Chapitre 1]{Halphen} says that a
rational map $F: \mathbb P^n \dashrightarrow \mathbb P^1$ with irreducible general fiber has at most two multiple fibers. In this section
we follow closely the exposition of Lins Neto \cite{HLN} to establish the  following generalization.

\begin{thm}\label{T:Halphen}
Let $X$ be a simply-connected compact K\"{a}hler manifold  and $F : X \dashrightarrow \mathbb P^1$ be meromorphic map.
If the general fiber of $F$ is irreducible then $F$ has at most two multiple fibers.
\end{thm}

We will say that a line bundle $\mathcal L$ is primitive if its Chern class $c_1(\mathcal L) \in H^2(X,\mathbb Z)$
generates a maximal rank $1$ submodule of $H^2(X,\mathbb Z)$. To adapt Lins Neto's proof of Halphen's Theorem to other manifolds we will
need the following lemma.

\begin{lemma}\label{L:890}
Let $X$ be a simply-connected compact complex manifold.
If $\mathcal L \in \Pic(X)$ is a primitive line bundle   on $X$ then the total space of $\mathcal L$ minus its
zero section is simply-connected.
\end{lemma}
\begin{proof}
Let $E$ be the total space of $\mathcal L$ minus its zero section.
As $E$ is a $\mathbb C^*$-bundle, we can use Gysin sequence
\[
H^1(X,\mathbb Z)  \to H^1(E,\mathbb Z) \to H^0(X,\mathbb Z) \stackrel{\wedge c_1(\mathcal L)}{\longrightarrow} H^2(X,\mathbb Z)
\]
to deduce that the fundamental group of $E$ is torsion.  If $E$ is not simply-connected then its universal covering is
a $\mathbb C^*$-bundle over $X$, and the associated line bundle divides $\mathcal L$.  This contradicts  the primitiveness of $\mathcal L$.
\end{proof}

\subsection*{Proof of Theorem \ref{T:Halphen}}Let $\mathcal L$ be a primitive line bundle and $k$ a positive integer such that
$\mathcal L^{\otimes k}=F^* \mathcal O_{\mathbb P^1}(1)$. If $E$ is the total space of the $\mathbb C^*$-bundle defined by $\mathcal L^*$ then
sections of $\mathcal L$ and its positive powers naturally  define  holomorphic functions on $E$. Moreover, if $f \in H^0(X,\mathcal L^{\otimes k})$ then
the element of $H^0(E,\mathcal O_E)$ determined by $f$, which we still denote by $f$, is homogeneous of degree $k$ with respect to
$\mathbb C^*$-action on $E$ given by fiberwise multiplication. In particular, if $R$ is the vector field on $E$ with flow  defining this $\C^*$-action
then we have the Euler identity
$
i_R df = k f \,
$
on $E$.

Now suppose $F: X \dashrightarrow \mathbb P^1$ has three multiple fibers, of multiplicity $p,q,r$.
Assume that they are over the points $[0:1], [1:0], [1:-1]$. Thus we can write $F= f^p/ g^q$ with
\begin{equation}\label{E:pqr}
f^p + g^q + h^r = 0,
\end{equation}
and $f^p,g^q,h^r \in H^0(X, \mathcal L^{\otimes k})$. If we interpret $f,g,h$ now as functions on
$E$ then taking the differential of the relation (\ref{E:pqr}) we get
$
pf^{p-1} df + q g^{q-1} dg + r h^{r-1} dh =0 \, .
$
Taking the wedge product first with $df$ and then with $dg$, we deduce the following equalities between holomorphic $2$-forms
\[
\frac{ df \wedge dg}{h^{r-1}} = \frac{ dg \wedge dh}{f^{p-1}}  = \frac{ df \wedge dh}{g^{q-1}}
\]
where we have deliberately omitted irrelevant constants.  If we contract these identities with $R$ we
get
\[
\omega = \frac{\frac{k}{p} f dg - \frac{k}{q} g df}{h^{r-1}}= \frac{ \frac{k}{q}g  dh - \frac{k}{r} h dg }{f^{p-1}}  = \frac{ \frac{k}{p} f  dh - \frac{k}{r} h df}{g^{q-1}}
\]
and $\omega$  can be interpreted as holomorphic section of $\Omega^1_X \otimes \mathcal L^{a}$ where
\[
a = \frac{k}{p} + \frac{k}{q} - \frac{(r-1)k }{r} = \frac{k}{q} + \frac{k}{r} - \frac{(p-1)k }{p} = \frac{k}{p} + \frac{k}{r} - \frac{(q-1)k }{q} \, .
\]
Since $X$ is K\"{a}hler and simply-connected,  $H^0(X,\Omega^1_X\otimes \mathcal L^{\otimes b})=0$ for any $b \le 0$. Thus $a>0$, and from this inequality we deduce that
\[
\frac{1}{p} +  \frac{1}{q} + \frac{1}{r}    =  1+ a > 1  \, .
\]
This implies that the triple $(p,q,r)$, after reordering,  must be one of the following: $(2,2,m), (2,3,3), (2,3,4)$, or $(2,3,5)$.

If $(p,q,r)$ belongs to this list then $\mathbb C^2 \setminus \{ 0\}$ is the universal covering of the surface   $S_{p,q,r} = \{ (x,y,z) \in \mathbb C^3 \setminus \{ 0 \} | x^p + y^q + z^r = 0 \}$. Moreover, the entries of the covering map $p = (F,G,H): \C^2 \setminus \{0 \} \to S_{p,q,r} $   are homogeneous
polynomials in two variables satisfying $F^p + G^q + H^r=0$, see \cite[Introduction]{HLN}.

Recall that $E$ is simply-connected according to Lemma \ref{L:890}. Since  the indeterminacy set  of $F$ has codimension two, the manifold
$E \setminus \{ f=g=h=0 \}$  is also simply-connected.
Therefore we can lift the map
\begin{align*}
\varphi : E \setminus \{ f=g=h=0 \} &\longrightarrow S_{p,q,r} \\
x & \longmapsto ( f, g, h ) \, .
\end{align*}
through the covering map $p$ to a map $\tilde \varphi : E \setminus \{ f=g=h=0 \}  \to \mathbb C^2 \setminus \{0\}$. The particular form of the covering map described above implies that $\tilde \varphi$ sends fibers of the $\mathbb C^*$-bundle $E$ to lines through origin
of $\C^2$, and therefore it descends to a rational map $G: X \dashrightarrow \P^1$ which  fits into the diagram below.
\[
\xymatrix{
& \mathbb P^1 \ar[d]
\\
X \ar@{-->}[ur]^G \ar@{-->}[r]^F  &  \mathbb P^1}
\]
Since the vertical arrow is not invertible,   the general fiber of $F$ is not irreducible. With this contradiction we conclude the proof.
\qed

\subsection{Codimension one stable foliations with first integrals} Having Theorem \ref{T:Halphen} at hand we are able
to give precisions about the structure of the first integrals of semi-stable foliations of codimension one having negative canonical bundle
on projective manifolds with rank one Picard group.

\begin{prop}\label{P:boaintegral}
Let $X$ be a simply-connected projective manifold with $\Pic(X)= \mathbb Z$ and  $\mathcal F$ be a codimension one foliation on $X$.
Suppose
\begin{itemize}
\item[(a)] $\mathcal F$ is semi-stable and $K\mathcal F <0$, or
\item[(b)] $\mathcal F$ is stable, has a rational first integral,  and $K\mathcal F = 0$.
\end{itemize}
Then $\mathcal F$ admits a rational first integral of the form $(f^p : g^q) : X \dashrightarrow \mathbb P^1$
where $p,q$ are relatively prime positive integers; and  $f,g$ are sections of line bundles $ \mathcal L_1, \mathcal L_2$ which satisfy
\[
\mathcal L_1^{\otimes p}  = \mathcal L_2^{\otimes q} \quad \text{ and } \quad N\mathcal F = \mathcal L_1 \otimes \mathcal L_2 .
\]
In particular $\mathcal F$ is defined by a logarithmic  $1$-form without divisorial components in its zero set.
\end{prop}
\begin{proof}
Let $F: X \dashrightarrow \mathbb P^1$ be a rational first integral for $\mathcal F$ with irreducible general fiber. Notice that the target has to be $\mathbb P^1$ since $\Pic(X) = \mathbb Z$.
Theorem \ref{T:Vistoli} implies that every fiber of $F$ is irreducible, and  Theorem \ref{T:Halphen} tells us that there are at most two non-reduced fibers. Assume that they
are over $0 , \infty \in \mathbb P^1$ and write $F^{-1}(0)= p H_0 $, $F^{-1}(\infty) = q H_{\infty}$ where $H_0$ and $H_{\infty}$ are reduced and irreducible hypersurfaces.
If we take the logarithmic $1$-form on $\mathbb P^1$ given in homogeneous coordinates by $dx/x - dy/y$ and we pull-back it by $F$ then the
resulting logarithmic $1$-form, which defines $\mathcal F$,  has polar divisor equal to $H_0 + H_{\infty}$ and empty zero divisor. Therefore $N \mathcal F = \mathcal O_X(H_0 + H_{\infty})$
and the rational function $F$ can be written as $(f^p: g^q)$ with $f \in H^0(X,\mathcal O_X(H_0))$, $g \in H^0(X,\mathcal O_X(H_{\infty}))$.  The proposition follows.
\end{proof}

\begin{cor}\label{C:boaintegral}
Let $\mathcal F$ be a semi-stable codimension one foliation on $\mathbb P^{n}$, $n \ge 3$. If $\deg(\mathcal F) < n-1$ then
$\mathcal F$ admits a rational first integral of form $(F^p:G^q)$ where $F$ and $G$ are homogeneous polynomials and $p,q$ are
relatively prime positive integers such that
$p \deg( F) = q \deg( G)$ and $\deg(\mathcal F) = \deg(F) + \deg(G) -2.$
\end{cor}

\subsection{Very negative foliations on Fano manifolds with rank one Picard group}
A projective manifold $X$ is Fano if its anticanonical bundle $-KX$ is ample.
Let $H$ be an ample generator of the Picard group of a Fano manifold with $\rho(X)=1$  ($\rho(X)$ is the rank of the Picard group of $X$). The {\bf index} of $X$, denoted by $i(X)$,
is defined through the relation $-KX = i(X) H$. The index of a Fano manifold of dimension $n$ is bounded by $n+1$ and the extremal cases are $\mathbb P^n$ ($i(X)=n+1$)
and hyperquadrics $Q^n \subset \mathbb P^{n+1}$ ($i(X) =n$), see \cite{KO}.

A codimension one foliation of degree one on $\mathbb P^n$ has canonical bundle $K\mathcal F$ equal to $\mathcal O_{\mathbb P^n}(2-n)$, see Example \ref{E:Pn}.
Our next result can be thought as a generalization of Jouanolou's classification of codimension one foliations of degree one on $\mathbb P^n$ \cite[Chapter I, Proposition  3.5.1]{EPA} to arbitrary Fano manifolds
with $\rho(X)=1$.

\begin{prop}\label{P:verynegative}
Let $X$ be a Fano manifold of dimension $n \ge 3$ and Picard number $\rho(X)=1$. Let  $H$ be an ample generator of $\Pic(X)$. If $\mathcal F$ is
a codimension one foliation on $X$ with $K \mathcal F= (2-n) H$ then   $\mathcal F$ is a foliation of degree one
on $\mathbb P^n$, or $\mathcal F$ is the restriction of a pencil of hyperplanes on $\mathbb P^{n+1}$ to a hyperquadric $Q^n$.
\end{prop}
\begin{proof}
Assume first that $\mathcal F$ is semi-stable. Theorem \ref{T:MBM} implies
$\mathcal F$ has a rational first integral.  Proposition \ref{P:boaintegral} implies $N\mathcal F \ge 2 H$.
Since $KX = K\mathcal F  - N\mathcal F$, it follows  that $KX \le -nH$. Therefore  $KX = -(n+1)H$, $N\mathcal F = 3 H$ and $X = \mathbb P^n$,  or $KX=-nH$, $N\mathcal F = 2 H$ and $X= Q^n$.
Proposition \ref{P:boaintegral} implies
$\mathcal F$ is a pencil of quadrics with a non-reduced member in the first case, and a pencil of hyperplane sections of $Q^n$ in the second case.

Suppose now that $\mathcal F$ is not semi-stable and let $\mathcal G$ be its maximal destabilizing foliation. Therefore
\[
- K \mathcal G = c_1 ( T \mathcal G) > \frac{-K\mathcal F}{\dim(\mathcal F)} \cdot \dim(\mathcal G) \ge (\dim(\mathcal G)- 1 )H \, .
\]
and, consequently,  $-K\mathcal G \ge \dim(\mathcal G) H$ and we can produce a non-zero section of $\wedge^{\dim(\mathcal G)} TX \otimes \mathcal O_X(-\dim(\mathcal G) H)$.
It follows from \cite[Theorem 1.2]{ADK} that $X= \mathbb P^n$ and $\mathcal G$ is a foliation of degree zero on $\mathbb P^n$.
These have been classified in \cite[Th\'{e}or\`{e}me 3.8]{CerveauDeserti}: a codimension $q$  foliation
of degree zero on $\mathbb P^n$ is defined by a linear projection from $\mathbb P^n$ to $\mathbb P^q$.
It follows that  $\mathcal F$ is the linear pull-back of a foliation of degree one on $\mathbb P^{n- \dim(\mathcal G)}$.
\end{proof}

In  \cite{AD} codimension one foliations with $K\mathcal F= (2-n) H$ are called codimension one del Pezzo foliations.

\section{Rough structure}\label{S:classifica}
The goal of this section is to prove the following  result.

\begin{thm}\label{T:rr}
Let $X$ be a Fano $3$-fold with $\Pic(X) = \mathbb Z$, and let $\mathcal F$ be a codimension one foliation on $X$
with trivial canonical bundle.
If $\mathcal F$ is not semi-stable then $X= \mathbb P^3$ and $\mathcal F$ is the linear pull-back of a degree two foliation on
$\mathbb P^2$. If $\mathcal F$ is semi-stable then  at least one of the following assertions holds true:
\begin{enumerate}
\item $T\mathcal F = \mathcal O_X \oplus \mathcal O_X$ and $\mathcal F$ is induced by an algebraic action;
\item there exists an algebraic action of $\mathbb C$ or $\mathbb C^*$ with non-isolated fixed points and  tangent to $\mathcal F$;
\item $\mathcal F$ is given by a closed rational $1$-form without divisorial components in its zero set.
\end{enumerate}
\end{thm}

\subsection{Division Lemma} To prove Theorem \ref{T:rr} we will use the following lemma.

\begin{lemma}\label{L:division}
Let $X$ be a projective $3$-fold, $\mathcal G$ be a one-dimensional foliation on $X$ with isolated singularities,
and $\mathcal F$ a codimension one foliation containing $\mathcal G$. If $H^1 ( X , KX \otimes K{\mathcal G}^{\otimes -2} \otimes N \mathcal F) =0$
then $T \mathcal F \cong T \mathcal G \oplus T \mathcal H$ for a suitable one-dimensional foliation $\mathcal H$.
\end{lemma}
\begin{proof}
Let $ v \in H^0 ( X, TX \otimes K{\mathcal G})$ be a twisted vector field defining $\mathcal G$. By hypothesis $v$ has isolated zeros.
Therefore (see for instance \cite[Exercise 17.20]{eisenbud})  contraction of differential forms with $v$ defines a resolution of the singular scheme $\mathrm{sing}(\mathcal G)$ of $\mathcal G$:
\[
0 \to \Omega^3_X \to \Omega^2_X \otimes K{\mathcal G} \stackrel{\Phi}{\longrightarrow } \Omega^1_X \otimes K{\mathcal G}^{\otimes 2} \to K{\mathcal G}^{\otimes 3} \to \mathcal O_{\mathrm{sing}(\mathcal G)} \to 0 \, .
\]
After tensoring by $N\mathcal F \otimes K _{\mathcal G} ^{ \otimes -2}$, we obtain from  the exact sequence above the following exact sequences
\[
0 \to Im \Phi \otimes  K{\mathcal G}^{\otimes -2} \otimes N \mathcal F \to \Omega^1_X \otimes N\mathcal F  \to K{\mathcal G} \otimes N   \mathcal F  \, ,
\]
and
\[ 0 \to \Omega^3_X \otimes K{\mathcal G}^{\otimes -2} \otimes N \mathcal F \to \Omega^2 \otimes N \mathcal F \otimes K{\mathcal  G}^{-1} \to  Im \Phi \otimes  K{\mathcal G}^{\otimes -2} \otimes N \mathcal F \to 0 \, .
\]

If $\omega \in H^0(X, \Omega^1_X \otimes N \mathcal F)$ defines $\mathcal F$ then, since $\mathcal F$ contains $\mathcal G$,
$\omega$ belongs to the kernel of
\[
H^0(X,\Omega^1_X \otimes N\mathcal F)  \to H^0(X,K{\mathcal G} \otimes N   \mathcal F ) \, .
\]
The first sequence
tells us that we can lift $\omega$ to $H^0(X,Im \Phi \otimes  K{\mathcal G}^{\otimes -2} \otimes N \mathcal F)$.
The second exact sequence, together with our cohomological hypothesis, ensures the existence of
$\theta \in H^0(X, \Omega^2_X \otimes N \mathcal F \otimes K{\mathcal  G}^{-1})$ such that $\omega = i_v \theta$.
The twisted $2$-form  $\theta$ defines the sought foliation $\mathcal H$. \end{proof}

\subsection{Automorphisms of a foliation}
Let $\mathcal F$ be a codimension one foliation on a projective manifold $X$. The automorphism group of $\mathcal F$, $\Aut(\mathcal F)$,
is the subgroup of $\Aut(X)$ formed by automorphisms of $X$ which send $\mathcal F$ to itself. It is a closed subgroup of $\Aut(X)$, and therefore
the connected component  of the identity is a finite dimensional connected Lie group. We will denote by $\mathfrak{aut}(\mathcal F)$ its Lie algebra,
which can be identified with a subalgebra of $\mathfrak{aut}(X) = H^0(X,TX)$. If $\mathcal F$ is defined by $\omega \in H^0(X, \Omega^1_X \otimes N \mathcal F)$
then we define the $\mathfrak{fix}(\mathcal F)$ as the subalgebra of $\mathfrak{aut}(\mathcal F)$ annihilating $\omega$, i.e.
\[
\mathfrak{fix}(\mathcal F) = \{ v \in \mathfrak{aut}(\mathcal F) \, | \, i_v \omega =0 \} \, .
\]
Notice that  $\mathfrak{fix}(\mathcal F)$ is nothing more than  $H^0(X,T \mathcal F)$.
We also point out  that $\mathfrak{fix}(\mathcal F)$ is an ideal of $\mathfrak{aut}(\mathcal F)$, and that
 subgroup $\Fix(\mathcal F) \subset \Aut(\mathcal F)$ generated by $\mathfrak{fix}(\mathcal F)$ is not necessarily  closed.

\begin{lemma}\label{L:campos}
The following assertions hold true:
\begin{enumerate}
\item If $\mathfrak{fix}(\mathcal F) = \mathfrak{aut}(\mathcal F)\neq 0 $  then there exists a non-trivial algebraic action with general orbit tangent to  $\mathcal F$.
\item If $\mathfrak{fix}(\mathcal F)\neq \mathfrak{aut}(\mathcal F)$  then $\mathcal F$ is generated by a closed rational $1$-form without divisorial components in its zero set.
\end{enumerate}
\end{lemma}
\begin{proof}
The connected component of the identity of $\Aut(\mathcal F)$ is closed. If $\mathfrak{fix}(\mathcal F)=\mathfrak{aut}(\mathcal F)$ then
$\Fix(\mathcal F)$ is also closed and therefore correspond to an algebraic subgroup of $\Aut(X)$. Item (1) follows.
To prove Item (2), let $v$ be a vector field in $\mathfrak{aut}(\mathcal F) - \mathfrak{fix}(\mathcal F)$.
If $\omega \in H^0(X, \Omega^1_X \otimes N \mathcal F)$ is a twisted $1$-form defining $\mathcal F$ then   \cite[Corollary 2]{Percy}
implies $(i_v \omega)^{-1} \omega$ is a closed meromorphic $1$-form. Since the singular set of $\omega$ has codimension at least two, the
same holds true for the zero set of  $(i_v \omega)^{-1} \omega$.
\end{proof}

\subsection{Proof of Theorem \ref{T:rr}}If $T \mathcal F$ is not semi-stable then Proposition \ref{P:unstable} implies the existence of a foliation by curves $\mathcal G$
tangent to $\mathcal F$ and with $\mu(T\mathcal G) > 0$.  According to Wahl's Theorem \cite{Wahl}, $X$ is isomorphic to $\mathbb P^3$ and $T\mathcal G = \mathcal O_{\mathbb P^3}(1)$. Thus $\mathcal G$ is
a foliation of degree zero and, consequently, its leaves are the lines through a point $p \in \mathbb P^3$. It follows
that $\mathcal F$ is a pullback  of foliation on $\P^2$ of degree two under the linear projection $\pi : \mathbb P^3 \dashrightarrow \mathbb P^2$
determined by $\mathcal G$.

Suppose now that $T \mathcal F$ is stable. If $\F$ is a foliation by algebraic leaves then  Proposition \ref{P:boaintegral} implies that   $\mathcal F$ is defined
by a logarithmic $1$-form without codimension one components in its zero set. Since Fano manifolds  are simply-connected \cite[Corollary 4.29]{Debarre},
every flat line bundle on $X$ is trivial.  Theorem \ref{T:truquesujo}  implies that  $\mathcal F$ is given by a closed rational
$1$-form without divisorial components in its zero set.

Finally, we will deal with the case where $T\mathcal F$ is strictly semi-stable. Now we have a foliation by curves $\mathcal G$ tangent to $\mathcal F$
with $T \mathcal G = \mathcal O_X$. In other words, $\mathcal G$ is induced by a vector field $v \in H^0(X,TX)$ with zeros of codimension at least two.
Notice that $\mathbb C v \subset \mathfrak{fix}(\mathcal F)$.

Suppose $\mathfrak{fix}(\mathcal F)= \mathfrak{aut}(\mathcal F)$.  If $\mathfrak{fix}(\mathcal F)= \mathbb C v$ then
we claim $\mathcal G$ is defined by an algebraic action of $\mathbb C$ or $\mathbb C^*$ with non-isolated fixed points.
Indeed Lemma \ref{L:campos} implies the existence of  an action of a one-dimensional Lie group tangent to $\mathcal F$.  Moreover, if
the action has only isolated fixed points then  we can apply Lemma \ref{L:division}  to deduce that the tangent bundle of $\mathcal F$ is
$\mathcal O_X \oplus \mathcal O_X$. Notice that the hypothesis of Lemma \ref{L:division} are satisfied since  $KX \otimes K{\mathcal G}^{\otimes -2} \otimes N \mathcal F=\mathcal O_X$
and $H^1(X,\mathcal O_X)=0$ for Fano manifolds.

If we still assume $\mathfrak{fix}(\mathcal F)= \mathfrak{aut}(\mathcal F)$
but now with  $\dim \mathfrak{fix}(\mathcal F)>1$ then,
as $v$ has  no divisorial components in its zero set, any two elements in it will generate $T\mathcal F$. Thus $T \mathcal  F= \mathcal O_X \oplus \mathcal O_X$ in
this case and $\mathcal F$ is defined by an algebraic action since $\Aut(\mathcal F)$ is closed.

Finally, if $\mathfrak{fix}(\mathcal F) \neq \mathfrak{aut}(\mathcal F)$ then Lemma \ref{L:campos} implies $\mathcal F$ is given by a closed meromorphic
$1$-form with zero set of codimension at least two.  \qed

\section{Foliations on the $3$-dimensional quadric}\label{S:Q3}
 We  will now classify the foliations with $K\F=0$ on the $3$-dimensional
quadric.  We start by presenting an example.

\begin{example}\label{E:affQ}
Identify $\mathbb P^4$ with the set of $4$ unordered points in $\mathbb P^1$. This identification
gives  a natural action of $PSL(2,\mathbb C) \simeq \Aut(\mathbb P^1)$ on $\mathbb P^4$.
Let    $p_0 \in \mathbb P^4$ be the point defined by the set $\{ 1,-1, i, -i \} \subset \mathbb P^1$.
The closure of the $PSL(2,\mathbb C)$-orbit of $p_0$ is a smooth quadric $Q^3 \subset \mathbb P^4$, see \cite{MukaiUmemura}.
This quadric can be decomposed as the union of three orbits of $PSL(2,\mathbb C)$: a closed orbit of dimension one
isomorphic to a rational normal curve $\Gamma_4$ of degree $4$ corresponding to points on $\mathbb P^1$ counted with
multiplicity $4$; an orbit $S$ of dimension two corresponding to  two distinct points  on $\mathbb P^1$, one with multiplicity three
and the other with multiplicity one (in more geometric terms this orbit is the tangent surface of $\Gamma_4$); and the open orbit of dimension three corresponding to $4$ distinct points isomorphic to $\{ 1,-1, i, -i \}$
The affine subgroup $\Aff(\mathbb C) \subset \Aut(\mathbb P^1)$ acts on $Q^3$ fixing the point $p=4\infty$, and defines on it a codimension one
foliation $\mathcal F$ with trivial tangent bundle. Notice that the singular set of $\mathcal F$ has three irreducible components: $\Gamma_4$; a twisted cubic $\Gamma_3$ corresponding to points of the form $3p +  \infty$; and a line corresponding to points of the form  $p + 3\infty$. Notice that the foliation $\F$ leaves invariant the surface
$S$ (which belongs to the linear system $|\mathcal O_{Q^3}(3)|$, see \cite[\S2.4]{jinv}), and that the quadratic cone through $p$ (which belongs to $|\mathcal O_{Q^3}(1)|$)
is the unique hyperplane section invariant by $\F$. This is sufficient to show that $\F$ is not in $\Rat(1,2)$ nor in $\Log(1,1,1)$. Indeed, \cite[\S 5.3.1]{vainsencher}
implies  that the image of the rational parametrization of $\Rat(1,2)$  defined in \S\ref{S:parametros} is closed. In particular, foliations in this component do not
leave irreducible elements of $|\mathcal O_{Q^3}(3)|$, like $S$, invariant. The rational parametrization of $\Log(1,1,1)$ do not have closed image, but if an element
is not on the image then the polar divisor of the corresponding closed rational $1$-form $\eta$ must $2H + H'$ or $3H$ where $H,H'$ are distinct elements of $|\mathcal O_X(1)|$.
According to the structure of closed rational $1$-forms  on projective manifolds \cite[appendix to Chapter VII]{Zariski},  in the first case $\eta$ is proportional to $h^2 h' \left( \frac{dh}{h} - \frac{dh'}{h'} + d  (\frac{f }{h})\right)$, and in the second case $\eta$ is proportional
to $h^3 d\left( \frac{g}{h^2}\right)$, where $f,h,h' \in H^0(Q^3,\mathcal O_{Q^3}(1))$ and $g \in H^0 (Q^3,\mathcal O_{Q^3}(2))$.  In the former case, the general leaf is not
algebraic while in the latter case the general leaf is an element of $|\mathcal O_{Q^3}(2)|$. In neither cases the foliation leaves an irreducible element of $|\mathcal O_{Q^3}(3)|$
invariant.
We conclude that $\F$ does not belong to $\Rat(1,2)$ nor to $\Log(1,1,1)$.
\end{example}

\begin{thm}\label{T:Quadricas}
The irreducible components of space of codimension one foliations with $K\mathcal F=0$ on the hyperquadric  $Q^3$ are
 $\Rat(2,1)$,  $\Log(1,1,1)$, and $\Aff$ (the general element is conjugated to the foliation presented in Example \ref{E:affQ}).
\end{thm}

 Theorem \ref{T:Quadricas} follows from Theorem \ref{T:rr} combined with  the next three propositions and Lemma \ref{L:CLN}.

\begin{prop}\label{P:QC*}
Let $\mathcal F$ be a codimension one foliation on $Q^3$ with $K \mathcal F=0$. If there exists an
algebraic $\mathbb C^*$-action with non-isolated fixed points that is tangent to $\mathcal F$ then $\mathcal F$ is given by a closed rational $1$-form
without divisorial components in its zero set.
\end{prop}
\begin{proof}
We can assume that $Q^3 \subset \mathbb P^4$ is given by the equation $\{ x_0^2 + x_1 x_2 + x_3 x_4 =0 \}$
and that $\mathbb C^* \subset \Aut(Q^3)$ is  a subgroup of the form
\[ \varphi_{\lambda}(x_0: x_1: x_2: x_3: x_4) = (x_0: \lambda^a x_1: \lambda^{-a} x_2 : \mu^b x_3: \mu^{-b} x_4),,\]
with  $a,b \in \mathbb N$ relatively prime,  since $\Aut(Q^3) = \mathbb P O(5, \mathbb C)$ has rank two.
If  $a$ and $b$ are distinct non-zero natural numbers then the fixed points of the action are isolated. Thus we have to analyze
only two cases: $(a,b)=(0,1)$ and $(a,b)=(1,1)$.

Let us start with the case $(a,b)=(0,1)$.  Consider the rational map
\begin{align*}
\Phi : \mathbb P^4 & \dashrightarrow \mathbb P( 1,1,1,2) \subset \mathbb P^{ 6} \, \\
(x_0:x_1:x_2:x_3:x_4) &\mapsto (x_0^2: x_0x_1 : x_0x_2 : x_1^2: x_1x_2: x_2^2:  x_3x_4 )  \, ,
\end{align*}
which identifies $\mathbb P( 1,1,1,2)$ with a cone over the Veronese surface in $\mathbb P^5$.
Notice that the quadric $Q^3$ is mapped to a hyperplane section of $\mathbb P(1,1,1,2)$ not passing
through the vertex $(0:0:0:0:0:0:1)$, which is of course isomorphic to $\mathbb P^2$.
We will  denote by $\Phi_0$ the induced rational map $\Phi_0: Q^3 \dashrightarrow \mathbb P^2$.
The general fiber of $\Phi_0$ is an  orbit of $\varphi$, and therefore the foliation $\mathcal F$
must be the pull-back of a foliation $\mathcal H$ on $\mathbb P^2$. Notice also that $\Phi_0^* \mathcal O_{\mathbb P^2}(1)$
is equal to $\mathcal O_{Q^3}(1)$. A simple computation shows that  the critical set of $\Phi_0$ has codimension greater than two.
Thus $\mathcal O_{Q^3}(3) = N \mathcal F = \Phi_0^* N \mathcal H$. It follows that $N\mathcal H= \mathcal O_{\mathbb P^2}(3)$, i.e.,
$\mathcal H$ has degree one. Since every foliation of degree one on $\mathbb P^2$ is induced by a closed meromorphic $1$-form
with isolated singularities \cite[Chapter 1, Section 2]{EPA} the proposition follows in this case.

Suppose now that $(a,b)=1$, and consider the rational map
\begin{align*}
\Phi : \mathbb P^4 & \dashrightarrow  \mathbb P^{4} \, \\
(x_0:x_1:x_2:x_3:x_4) &\mapsto (x_0^2: x_1x_2 : x_1x_4 :  x_2x_3: x_3x_4 )  \, .
\end{align*}
Its image is contained in a cone over a smooth quadric surface in $\mathbb P^3$. The quadric $Q^3$ is mapped
into a smooth hyperplane section of this cone which is isomorphic to $\mathbb P^1 \times \mathbb P^1$.
If we denote by $\Phi_0 : Q^3 \dashrightarrow \mathbb P^1 \times \mathbb P^1$ the induced rational map then
$\Phi_0^* \mathcal O_{\mathbb P^1 \times \mathbb P^1}(c,d) = \mathcal O_{Q^3}( c+ d)$. The only divisorial component of the critical set of $\Phi_0$ is the intersection
of the hyperplane $\{x_0=0\}$ with $Q^3$. The image of this critical set is a $(1,1)$ curve $C$
in $\mathbb P^1 \times \mathbb P^1$.
If $\mathcal G$ is a foliation on $\mathbb P^1 \times \mathbb P^1$ with
normal bundle $\mathcal N \mathcal G = \mathcal O_{\mathbb P^1\times \mathbb P^1} (c,d)$ then
\[
N \Phi_0^* \mathcal G = \left\{ \begin{array}{lcl}
\mathcal O_{Q^3} ( c+ d )&\text{ if }& \text{ $C$ is not $\mathcal G$-invariant }\\
\mathcal O_{Q^3} (c+d -1 )&\text{ if }& \text{ $C$  is  $\mathcal G$-invariant }.
                    \end{array} \right.
\]
Therefore if $\mathcal F = \Phi_0^* \mathcal G$ and $N \mathcal F= \mathcal O_{Q^3}(3)$ then
$c=d=2$ and $C$ is $\mathcal G$-invariant. A foliation $\mathcal G$ on $\mathbb P^1 \times \mathbb P^1$
with $N \mathcal G = \mathcal O_{\mathbb P^1 \times \mathbb P^1}(2,2)$ is given by a closed rational $1$-form
$\omega = \pi_1^* \omega_1 + \pi_2^*\omega_2$  where $\pi_1,\pi_2: \mathbb P^1 \times \mathbb P^1 \to \mathbb P^1$
are the natural projections and the $1$-forms $\omega_i$ have polar set of degree two. Since the $(1,1)$-curve
$C$ is $\mathcal G$-invariant, we must have $\omega_1 = - \omega_2 = dx_0/x_0 - dx_1/x_1$ in a suitable choice of coordinates
where $C = \{ x_0y_1 - y_0 x_1= 0 \}$. Therefore
\[
\omega = \frac{dx_0}{x_0} - \frac{dx_1}{x_1} - \frac{dy_0}{y_0} + \frac{dy_1}{y_1} \,  .
\]
Notice that $\omega$ is proportional to
\[
\alpha = \left( \frac{d (x_0y_1 - y_0x_1)}{x_0y_1 - y_0x_1} - \frac{dx_0 }{x_0} - \frac{dy_1}{y_1} \right) .
\]
and the pull-back of $\alpha$ under $\Phi_0$ is  closed $1$-form without divisorial components in its zero set.
\end{proof}

\begin{prop}\label{P:QC+}
Let $\mathcal F$ be a codimension one foliation on $Q^3$ with $K \mathcal F=0$. If $\mathcal F$ is tangent to an
algebraic $\mathbb C$-action with non-isolated fixed points then $\mathcal F$ is given by a closed rational $1$-form
without divisorial components in its zero set.
\end{prop}
\begin{proof}
Let $\varphi : \mathbb C \times Q^3 \to Q^3$ be an algebraic $\mathbb C$-action. As such, it must be of the form
$\varphi(t)= \exp(t \cdot n)$ where $n$ is a nilpotent element of the Lie algebra $\mathfrak{aut}(Q^3) = \mathfrak{so}(5, \mathbb C)$. In $\mathfrak{so}(5,\mathbb C)$
there are exactly three  $\Aut(Q^3) = \mathbb P O(5,\mathbb C)$-conjugacy classes of non-zero nilpotent elements. The Jordan normal forms
of the corresponding matrices in $End(\mathbb C^5)$ have: (1) only one Jordan block of order $5$; (2) one Jordan block of order $3$
and two trivial (order one) Jordan blocks ; or (3) two Jordan blocks of order $2$ and one trivial Jordan block.

The action in case (1) has isolated fixed points and  is excluded by hypothesis. To deal with case (2) we can assume that $n = x_1 \frac{\partial }{\partial x_0} + x_2 \frac{\partial }{\partial x_1}$
and that the quadric $Q^3$ is $\{ x_1 ^2 - 2x_0 x_2 + x_3^2 + x_4^2 =0 \}$. The general fiber of the rational map
\begin{align*}
\Phi : \mathbb P^4 & \dashrightarrow \mathbb P^6 \, \\
(x_0 : x_1  : x_2: x_3: x_4 ) & \mapsto ( x_1^2 - 2x_0 x_2 : x_2^2 : x_2x_3: x_2x_4: x_3^2: x_3x_4: x_4^2 )
\end{align*}
coincides with an orbit of $\varphi$, and  sends $\mathbb P^4$ to a cone over the second Veronese embedding of $\mathbb P^2$.
The image of the quadric $Q^3$ avoids the vertex of this cone and is isomorphic to $\mathbb P^2$. Moreover, the critical set of
$\Phi_0 : Q \dashrightarrow \mathbb P^2$ (the restriction of $\Phi$ to $Q$) has no divisorial components. Therefore
every foliation $\mathcal F$ on $Q^3$ tangent to $\varphi$ is of the form $\Phi_0^* \mathcal G$  for some foliation
on $\mathbb P^2$ and its normal bundle satisfies $N \mathcal F = \Phi_0^* N \mathcal G$. Since $\Phi_0^* \mathcal O_{\mathbb P^2}(1) = \mathcal O_Q(1)$,
it follows that $\mathcal F$ is the pull-back of a foliation $\mathcal G$ on $\mathbb P^2$ of degree one and, as such, is given
by a  closed $1$-form without zeros of codimension one \cite[Chapter 1, Section 2]{EPA}.

Case (3) is very similar to case (2). Now the vector field $n$ is of the form $x_1 \frac{\partial}{\partial x_0} + x_3 \frac{\partial}{\partial x_2}$,
the quadric is $Q=\{ x_0x_3 - x_1x_2 + x_4^2 =0 \}$ and
the quotient map is
\begin{align*}
\Phi : \mathbb P^4 & \dashrightarrow \mathbb P^6 \, \\
(x_0 : x_1  : x_2: x_3: x_4 ) & \mapsto (  x_0x_3 - x_1x_2 : x_1^2 : x_1x_3: x_1x_4: x_3^2: x_3x_4: x_4^2 ) \, .
\end{align*}
The restriction of $\Phi$ to $Q$ has critical set of codimension at least two, and therefore the conclusion is the same: $\mathcal F$
is the pull-back under $\Phi_{|Q}$ of a foliation on $\mathbb P^2$ of degree one, and as
such is defined by a closed rational $1$-form with zeros of codimension at least two.
\end{proof}

\begin{prop}
Let $\mathcal F$ be a codimension one foliation on $Q^3$ with trivial canonical bundle.
Suppose that $\mathcal F$ is induced by an algebraic action of a two dimensional  Lie subgroup
of $\Aut(Q^3)$. Then $\mathcal F$ is defined by a closed $1$-form without zeros of codimension one, or $\mathcal F$ is conjugated to
the foliation presented in Example \ref{E:affQ}.
\end{prop}
\begin{proof}
Let $G \subset \Aut(Q^3)$ be the subgroup defining $\mathcal F$, and $\mathfrak g \subset \mathfrak{so}(5,\mathbb C)$ the corresponding
Lie subalgebra. If $G$ is abelian then it must be
of the form $\mathbb C^*\times \mathbb C^*$, $\mathbb C \times \mathbb C^*$, or $\mathbb C \times \mathbb C$.
In the first case every element in $\mathfrak g$, the Lie algebra of $G$, is a semi-simple element of $\mathfrak{so}(5,\mathbb C)$.
Since the rank of  $\mathfrak{so}(5,\mathbb C)$ is two, $\mathfrak g$ is a Cartan subalgebra of $\mathfrak{so}(5,\mathbb C)$. Therefore,
we can find $\mathbb C^* \subset G$ inducing an algebraic action with non-isolated fixed points tangent to $\mathcal F$. We can apply Proposition \ref{P:QC*}
to conclude that $\mathcal F$ is induced by a  closed $1$-form without codimension one zeros. In the two remaining cases, $\mathfrak g$ contains a nilpotent
element $n$ which defines an algebraic subalgebra $\mathbb C \subset G$. If the corresponding action has non-isolated fixed
points then Proposition \ref{P:QC+} implies $\mathcal F$ is defined by a closed rational $1$-form without divisorial components in its zero set.

If the corresponding action has only isolated fixed points then we can assume that $Q$ is defined by the quadratic form $q= x_2^2 - 2 x_1 x_3 + 2x_0 x_4 $ and that
$n$, seen as an element of $\mathfrak{so}(q,\mathbb C)$,
has only one Jordan block of order $5$. The centralizer $C(n)$ of $n$ in $\mathfrak{so}(q,\mathbb C)$   is thus formed by nilpotent matrices of the form
\[
\left(
\begin{array}{ccccc}
  0 & \alpha & 0 &\beta&0\\
  0 & 0 & \alpha &0&\beta\\
  0 & 0 & 0& \alpha &0\\
  0 & 0 & 0&0& \alpha \\
      0 & 0 & 0&0& 0 \\
\end{array}
\right) \, .
\]
In particular, since $\mathfrak{g} \subset C(n)$, $\mathfrak{g}$ contains another nilpotent element which defines a $\mathbb C$-action with non-isolated
fixed points. Proposition \ref{P:QC+} implies $\mathcal F$ is defined by a  closed $1$-form without codimension one zeros.

\medskip

Suppose now that $G$ is not abelian. Its Lie algebra $\mathfrak{g}$ is isomorphic to the affine Lie algebra $\mathbb C x \oplus \mathbb C y$ with the
relation $[x,y]=y$. This relation implies that  $y$ is a nilpotent element of $\mathfrak{so}(5,\mathbb C)\subset \mathfrak{sl}({5,\mathbb C})$. As before,
using Proposition \ref{P:QC+}, we can reduce to the case where $y$ is in Jordan normal form and
has only one Jordan block of order $5$. The elements  $x \in \mathfrak{so}(q,\mathbb C)$ satisfying $[x,y]=y$ are of the form
\[
\left(
\begin{array}{ccccc}
  2 & \alpha &0&\beta&0\\
  0 & 1 & \alpha &0&\beta\\
  0 & 0 & 0& \alpha &0\\
  0 & 0 & 0&-1& \alpha \\
          0 & 0 & 0&0& -2 \\
\end{array}
\right) \, .
\]
After one last conjugation by an element of $SO(q,\mathbb C)$ we can suppose that $\beta=0$.  We have just proved that up to automorphisms of $Q^3$
 there is only one
foliation defined by an algebraic action of an algebraic subgroup $G \subset \Aut(Q^3)$ which is not invariant an algebraic action of a one-dimensional Lie group with non-isolated fixed points. Therefore
it must be  the foliation described in Example \ref{E:affQ}.
\end{proof}

\section{Foliations on projective spaces}\label{S:P3}
Let us recall  the main result of \cite{CerveauLinsNetoAnnals}.

\begin{thm}\label{T:CLN}
The irreducible components of the space of codimension one foliations of degree $2$ on $\mathbb P^n$, $n\ge 3$,  are
$\Rat(2,2)$,  $\Rat(1,3)$,  $\Log(1,1,1,1)$,  $\Log(2,1,1)$, $\PBL(2)$, and  $Exc(2)$.
\end{thm}

For every $d\ge 2$ the elements of $\PBL(d)$ are pull-backs of degree $d$ foliations
on $\mathbb P^2$ under linear projections $\pi : \mathbb P^n \dashrightarrow \mathbb P^2$.
The general element
of $Exc(2)$ is a  pull-back under a linear projection $\P^n\dashrightarrow \P ^3$ of the foliation on $\P^3\simeq \P \C_3[x]$   induced by the natural action of the affine group $\mathrm{Aff}(\C)$ on the projectivization of space of polynomials of degree at most three.
For a detailed description of the general element of $Exc(2)$ on $\P^3$ see \cite{jinv}.

Following essentially the same steps as used to proof Theorem \ref{T:Quadricas} one can recover Theorem \ref{T:CLN} for $n=3$  without using Dulac's classification
of quadratic centers \cite{Dulac} (see also \cite[Theorem 7]{CerveauLinsNetoAnnals}) and bypassing the computer-assisted calculations used to prove
\cite[Theorem E']{CerveauLinsNetoAnnals}. Nevertheless, to establish analogues of Propositions \ref{P:QC*}, and \ref{P:QC+} following the same strategy as above, one would
be lead to a lengthy  case-by-case analysis which we have chosen to not carry out  in detail here, but which can be found in version 2 of \cite{LPT} at the  arXiv. Instead, we will present below a proof of Theorem \ref{T:CLN} for $n>3$ assuming that
it holds true for $n=3$, as it may serve as a model to extend the results of the previous section to higher dimensional hyperquadrics. We start with the  classification of degree one foliations of arbitrary codimension on $\P^n$, a result
of independent interest which is used in  \cite{AD}.

\subsection{The space of foliations on $\mathbb P^n$ of degree one (arbitrary codimension)}\label{S:deg1}
We already recalled the classification of the foliations of degree zero in the proof of Proposition \ref{P:verynegative}: a codimension $q$  foliation
of degree zero on $\mathbb P^n$ is defined by a linear projection from $\mathbb P^n$ to $\mathbb P^q$.
The classification of foliations of degree one can be easily deduce from Medeiros classification of locally decomposable
integrable homogeneous $q$-forms of degree one (\cite[Theorem A]{Airton}) as we show below.

\begin{thm}\label{T:deg1}
If  $\mathcal F$ be a foliation of degree $1$ and codimension $q$ on $\mathbb P^n$ then we are in one of following cases:
\begin{enumerate}
\item  $\mathcal F$ is defined a dominant rational map $\mathbb P^n \dashrightarrow \mathbb P(1^{q}, 2)$ with irreducible general fiber
determined by  $q$   linear forms and one quadratic form; or
\item  $\mathcal F$ is the linear pull back of a foliation of induced by a global holomorphic vector field on $\mathbb P^{q+1}$.
\end{enumerate}
\end{thm}
\begin{proof}
We start by recalling \cite[Theorem A]{Airton}: if $\omega$ is a locally decomposable integrable homogeneous  $q$-form of degree $1$ on $\mathbb C^{n+1}$  then
\begin{enumerate}
\item[(a)] there exist $q-1$   linear forms $L_1, \ldots, L_{q-1}$  and a quadratic form $Q$ such that
$
\omega = d L_1 \wedge \cdots \wedge d L_{q-1} \wedge d Q \, ;
$
or
\item[(b)] there exist a linear projection $\pi : \mathbb C^{n+1} \to  \mathbb C^{q+1}$, and
a locally decomposable integrable homogeneous  $q$-form $\widetilde \omega$ of degree $1$ on $\mathbb C^{q+1}$ such that  $\omega = \pi^* \widetilde \omega$.
\end{enumerate}

Let $\omega$ be a locally decomposable, integrable homogeneous $q$-form on $\mathbb C^{n+1}$  defining $\mathcal F$. Since $\mathcal F$ has degree $1$, the degree of the coefficients of $\omega$ is $2$. It is immediate from the definitions that the differential of a locally decomposable integrable $q$-form is also
locally decomposable and integrable. Therefore we can apply  \cite[Theorem A]{Airton} to $d\omega$. To recover information about $\omega$ we will use that $i_R \omega =0$
implies
$  i_R d\omega = (q+ 2)  \cdot \omega \, .$

If $d \omega$ is in case (a)
then $d \omega$ is the pull-back of $dx_0\wedge\cdots\wedge dx_q$ under the map $$ \mathbb C^{n+1} \ni (x_0, \ldots, x_n) \mapsto (L_1, \ldots, L_q, Q)\in \mathbb C^{q+1},$$ and
$(q+2)\omega= i_R d\omega$ is the pull-back of $i_{R(1^q,2)} dy_0\wedge \cdots \wedge dy_q$  where $R(1^q,2) = y_0 \frac{\partial}{\partial y_0} + \cdots + y_q \frac{\partial}{\partial y_q} + 2 y_{q+1} \frac{\partial}{\partial y_{q+1}}$. We are clearly in case (a) of the statement with rational map from $\mathbb P^n\dashrightarrow  \mathbb P(1^{q}, 2)$ described in homogeneous coordinates as above. It still remains to check that the general fiber is irreducible. As $\omega$ has zero set of codimension at least two, the same holds true for $d\omega$ and consequently the
map considered does not ramify in codimension one. Since $\mathbb P^n$ is simply-connected, the irreducibility of the general fiber follows.

If  $d \omega$ is in case (b) then, in suitable coordinates,  $d \omega$ depends only on $q+2$ variables, say $x_0, \ldots, x_{q+1}$. Being a $(q+1)$-form with coefficients of degree $1$, there exists a linear vector field $X$ such that $d\omega = i_X dx_0 \wedge \cdots \wedge dx_{q+1}$. The result follows.
\end{proof}

\begin{cor}
The space of foliations of degree $1$ and codimension $q$ on $\mathbb P^n$ has two irreducible components.
\end{cor}

 \subsection*{Proof that Theorem \ref{T:CLN} for  $n=3$ implies Theorem \ref{T:CLN} for  $n>3$ }

Notice that when $n>3$, a foliation of degree two has negative canonical bundle. Thus, if $\F$ is semi-stable
Proposition \ref{P:boaintegral} implies that $\F$ is either a pencil of quadrics or a pencil of cubics having a hyperplane with multiplicity three as a member.

Suppose now that $\F$ is unstable and let $\G$ be its maximal destabilizing foliation. Recall from Example
\ref{E:concreto}
that
\[
\frac{\deg(\G)}{\dim(\G)} < \frac{\deg(\F)}{\dim(\F)} \, .
\]
Therefore $\deg(\G)<2$. If $\G$ has degree zero then $\F$ is a linear pull-back  of a foliation
of degree two on a lower-dimensional projective space and we can proceed inductively. Suppose now that the degree of $\G$ is one. The classification of foliations of degree one, Theorem \ref{T:deg1}, implies that the
semi-stable foliations of degree one are either defined by a rational map to $\mathbb P(1^q,2)$ or
have dimension one. The maximal destabilizing foliation $\G$, which is semi-stable by definition, does not fit into  the former case as  we would have $1 < \deg(\F) /\dim(\F)$. Thus $\G$ must be defined by a rational map to
$\mathbb P(1^q,2)$. It is not hard to verify that in this case the foliation $\F$ must be in the component $\Log(1,1,2)$. \qed

\section{Foliations on Fano $3$-folds of index two}\label{S:i2}

We know turn our attention to Fano $3$-folds of index two.  Unlikely in the cases where the index is four (projective space) or
three (quadric), these $3$-folds have moduli as the $3$-folds of index one. As will be seen below the space of foliations with
$K\F=0$  on Fano $3$-folds of index one or two  behaves rather uniformly with respect to the moduli, with only two exceptions. The exceptions
are the  quasi-homogeneous $PSL(2,\C)$--manifolds of index one and two.

Let $X$ be a Fano $3$-fold with $\Pic(X) = \mathbb Z H$ and index $i(X)=2$ which means, by definition, $-KX= 2H$. In this case the
classification is very precise (see \cite{JR} and references therein) and says that $X$ is isomorphic to a $3$-fold fitting in one of the following classes:
\begin{enumerate}
\item{$H^3=1$.} Hypersurface of degree $6$ in   $\mathbb P(1,1,1,2,3)$;
\item{$H^3=2$.} Hypersurface  of degree $4$ in   $\mathbb P(1,1,1,1,2)$;
\item{$H^3=3$.} Cubic in $\mathbb P^4$;
\item{$H^3=4$.} Intersection of two quadrics in $\mathbb P^5$;
\item{$H^3=5$.} Intersection of the Grassmannian $\Gr(2,5) \subset \mathbb P^9$ with a $\mathbb P^6$.
\end{enumerate}
Although not evident from the description above, the $3$-folds falling in class (5) are all isomorphic to
a $3$-fold $X_5 \subset \mathbb P^6$. In \cite{MukaiUmemura}  $X_5$ is described as an equivariant compactification
of  $\Aut(\mathbb P^1)/ \Gamma$ where $\Gamma$ is the octahedral group. Explicitly, if we consider the  point $p_0 \in \Sym^6 \mathbb P^1$  defined
by the polynomial $xy(x^4-y^4)$ then $X_5$ is the closure of the $\Aut(\mathbb P^1)$-orbit of $p_0$ under the natural action.

\begin{thm}\label{T:i2}
Let $X$ be a  Fano $3$-fold with $\Pic(X) = \mathbb Z H$ and index $i(X)=2$. If $X \neq X_5$ then the space of
codimension one foliations on $X$ with trivial canonical bundle is  irreducible.  If $X = X_5$ then
the space of codimension one foliations on $X$ with trivial canonical bundle has two irreducible components.
\end{thm}

As we will see from its proof the result is much more precise as it describes quite precisely the irreducible
components. We summarize the description in the Table below.

\medskip

\begin{tabular}{|l|c|c|}
\hline
{\bf Manifold} & {\bf Irreducible component} & {\bf dimension} \\
\hline\hline
Hypersurface of degree $6$
in   $\mathbb P(1,1,1,2,3)$&  $\Rat(1,1)     $ & $2$ \\
\hline
Hypersurface  of degree $4$
 in $\mathbb P(1,1,1,1,2)$ &  $\Rat(1,1) $ & $ 4$ \\
 \hline
  Cubic in $\mathbb P^4$ &  $\Rat(1,1) $ & $6$ \\
  \hline
  Intersection of quadrics in $\mathbb P^5$  &  $\Rat(1,1) $ & $8$ \\
  \hline
\multirow{2}{*}{$X_5$} &  $\Rat(1,1)  $ & $10$ \\
  \cline{2-3}
                  &$ \Aff $ & $1$
      \\          \hline
  \end{tabular}

\begin{lemma}\label{L:X5}
The dimension of $H^0(X_5, TX_5)$ is $3$, and every $v \in H^0(X_5,TX_5)$ has isolated singularities.
\end{lemma}
\begin{proof}
Let $\Sigma$ be the variety of lines contained in $X_5$. According to \cite{FurushimaNakayama}, $\Sigma$ is isomorphic to $\mathbb P^2$.
The induced action of $\Aut(\mathbb P^1)$ on it has one closed orbit isomorphic to a conic $C \subset \mathbb P^2$, and  one open orbit
isomorphic to $\mathbb P^2 \setminus C$. It can be identified with the natural action of $\Aut(\mathbb P^1)$ in $\Sym^2 \mathbb P^1 \simeq \mathbb P^2$.
If an automorphism of $X_5$ induces the identity on $\Sigma$ then it must be identity since through every point of every line contained in $X_5$ passes
at least another line, loc. cit. Corollary 1.2. This suffices to show that  $h^0(X_5,TX_5)=3$.

Let now $v \in H^0(X_5,TX_5)$ be a non-zero vector field, and $H = \exp(\mathbb Cv)\subset \Aut(X_5)$ be the one-parameter
subgroup generated by it. The description of the induced action of $\Aut(X)$ on $\Sigma$ implies that the induced action of $H$ on
$\Sigma$ has isolated fixed points. Therefore, if the zero set of $v$ has positive dimension then it must be contained
in a finite union of lines. If we take $\ell$ as one of these lines then  the action of $H$ on $\Sigma$ would fix all the lines intersecting
$\ell$. This contradicts the description of the induced action of $\Aut(X)$ on $\Sigma$.
\end{proof}

\begin{lemma}\label{L:v00}
Let $\mathbb P= \mathbb P(q_0, q_1,q_2,q_3,q_4)$ be a well-formed weighted projective space of dimension four with $q_0\le q_1 \le q_2 \le q_3 \le q_4$,
and $X \subset \mathbb P$ be a smooth hypersurface. If $\deg(X) \ge q_2 + q_3 + q_4$
then $h^0(X,TX)=0$.
\end{lemma}
\begin{proof}
Set $d=\deg(X)$ and $Q= \sum_{i=0}^4 q_i$. By \cite[Theorem 3.3.4]{Dolg}, $\Omega^3_X = \mathcal O_X(d- Q)$.
Consequently $TX = \Omega^2_X \otimes \mathcal O_X(Q- d)$.
From the long exact sequence associated to
\[
0 \to \Omega^1_X\otimes N^*_X \otimes \mathcal O_X(Q-d) \to {\Omega^2_{\mathbb P}}_{|X}(Q-d) \to \Omega^2_X(Q-d) \to 0 \,
\]
we see that $h^0(X,TX)=0$ when  $h^0(X,{\Omega^2_{\mathbb P}}_{X}(Q-d))= h^1(X,\Omega^1_X (Q-2d))=0$.

To compute $h^1(X,\Omega^1_X (Q-2d))$, consider the conormal sequence of $X \subset \mathbb P$ tensored by $\mathcal O_X(Q-2d)$
\[
0 \to N^*_X( Q- 2d) \to {\Omega^1_{\mathbb P}}_{|X}(Q-2d) \to \Omega^1_X(Q-2d) \to 0 \, .
\]
On the one hand, as the intermediary cohomology of $\mathcal O_X(n)$ vanishes for every $n \in \mathbb Z$ \cite[Theorem 3.2.4 (iii)]{Dolg},
$h^2(X,N^*_X (Q-2d)) =  h^2(X,\mathcal O_X (Q-3d)) =0$. On the other hand $H^1(X, {\Omega^1_{\mathbb P}}_{X}(Q-2d))$ can be computed with the
exact sequence
\[
0 \to   \Omega^1_{\mathbb P}(Q-3d)  \to \Omega^1_{\mathbb P} ( Q - 2d) \to {\Omega^1_{\mathbb P} }_{|X} ( Q - 2d) \to 0 \, .
\]
Now \cite[Theorem 2.3.2]{Dolg} tell us that
$h^2(\mathbb P, \Omega^1_{\mathbb P} (n)=0$  for every $n \in \mathbb Z$, and $h^1(\mathbb P, \Omega^1_{\mathbb P} (n))=0$ if and only if $n \neq 0$.
But $d \ge q_2+q_3 +q_4$, as we have assumed, implies $2d > Q$. Thus $h^1(X, {\Omega^1_{\mathbb P}}_{X}(Q-2d)) = 0$ as wanted.

It remains to show that $h^0(X, {\Omega^2_{\mathbb P}}_{|X} ( Q - d ))=0$. To do it, consider the exact
sequence
\[
0 \to \Omega^2_{\mathbb P}(Q-2d) \to \Omega^2_{\mathbb P}(Q-d) \to {\Omega^2_{\mathbb P}}_{|X} ( Q - d ) \to 0 \, .
\]
The vanishing of $h^1(\mathbb P, \Omega^2_{\mathbb P}(Q-2d))$ is assured by \cite[Theorem 2.3.4]{Dolg}. Finally, \cite[Corollary 2.3.4]{Dolg}
implies $h^0(\mathbb P , \Omega^2_{\mathbb P}(Q-d)) \neq 0$ if and only if
\[
d < Q - q_0 - q_1 .
\]
The lemma follows.
\end{proof}

Lemma \ref{L:v00} together with the classification of Fano $3$-folds of index two imply the following corollary.

\begin{cor}\label{C:nofields}
If $X$ is a Fano $3$-fold with $\rho(X)=1$ and $i(X)=2$ then $h^0(X,TX) \neq 0$ if and only if
$X$ is isomorphic to $X_5$.
\end{cor}

\subsubsection*{{\bf Proof of Theorem \ref{T:i2}}} Let $X$ be a Fano $3$-fold of index two with $\Pic(X)= \mathbb Z  \cdot H$, where
$H$ is an ample divisor, and $\mathcal F$ a codimension one foliation
on $X$ with $K \mathcal F=0$. If $H^3 \le 4$ then Corollary \ref{C:nofields} implies $X$ has no vector fields. Therefore by Theorem \ref{T:rr}
any foliation on $X$ with $K\mathcal F=0$  is given by a closed $1$-form without codimension one zeros and with polar divisor linearly equivalent
to $2H$. The result follows Lemma \ref{L:CLN}.
Notice that the dimension of $H^0(X,\mathcal O_X(H))$ is equal to $H^3 + 2$,  \cite[Chapter V, Exercise 1.12.6]{kollar}.

Suppose now that $H^3=5$, i.e., $X=X_5$. Lemma \ref{L:X5} implies that every algebraic action of $\mathbb C$ or $\mathbb C^*$ has isolated
fixed points. Theorem \ref{T:rr} tells us that a foliation on $X_5$ with trivial canonical bundle is either induced
by an algebraic action of two dimensional Lie group or is given by a  closed $1$-form without codimension one zeros and with polar divisor linearly equivalent
to $2H$. The Lie algebra of regular vector fields on $X_5$ is isomorphic to $\mathfrak{sl}(2,\mathbb C)$ (Lemma \ref{L:X5}) and
the two dimension subalgebras are all $PSL(2,\mathbb C)$-conjugated, and  isomorphic to the affine Lie algebra
$\mathfrak{aff}(\mathbb C)$. Let $\mathcal F$ be a foliation of $X_5$ determined by any of the affine Lie algebras contained into
$\mathfrak{sl}(2,\mathbb C)$.  The induced action of $\Aff(\mathbb C)\subset \Aut(X)$ on $\mathbb PH^0(X, \Omega^1_X(H))$ has
only one fixed point, therefore $\Aff(\mathbb C)$ is tangent to only one hyperplane section of $X_5 \subset \mathbb P^6$.
It follows that $\mathcal F$ is not defined by a closed $1$-form without codimension one zeros since in this case the action would have to
preserve a pencil of hyperplane sections. As there is a smooth $\mathbb P^1$ of affine Lie subalgebras of $\mathfrak{sl}(2,\mathbb C)$
we conclude that  the space of foliations on $X_5$ with $K\mathcal F=0$ has two disjoint irreducible components: one corresponding
to foliations defined  by closed $1$-forms and the other defined by affine subalgebras of $\mathfrak{aut}(X_5)$. Notice that they are
both smooth, with the second one corresponding to a closed orbit of $\Aut(X_5)$ in $\mathbb P H^0(X_5,\Omega^1_{X_5}(2H))$.
\qed

\section{Foliations on Fano $3$-folds of index one}\label{S:i1}
Most of the work for the classification of foliations with $K\F=0$ on Fano $3$-folds with $\Pic(X)= \Z$ and of index one has
already been done by Jahnke and Radloff in \cite{JR}.  In \cite[Proposition 1.1]{JR} it is
proved that $h^0(X, \Omega^1_X(1)) \neq 0$ implies that the genus of $X$, which by definition is $g(X) = h^0(X,-KX) +2 = \frac{1}{2}KX^3 + 1$, is $10$ or $12$. This considerably reduces the amount of work to
prove the final bit in the classification of foliations with $K\F=0$ on Fano $3$-folds with rank one Picard group.

\begin{thm}\label{T:i1}
If $\mathcal F$ is a codimension one foliation with trivial canonical bundle  on a Fano $3$-fold with $\Pic(X)=\mathbb Z$ and $i(X)=1$ then
$X$ is the Mukai-Umemura $3$-fold and $\mathcal F$ is induced by an algebraic action  of the affine group.
\end{thm}

Recall that the Mukai-Umemura $3$-fold is the quasi-homogeneous $3$-fold obtained
by the closure of the $\Aut(\P^1)$-orbit  of the point of in $\Sym^{12}\P^1$ determined by the polynomial  $xy(x^{10} + 11x^5y^5 + y^{10})$.
It is an equivariant compactification of the quotient of $\Aut(\P^1)$ by the icosahedral group $A_5$, \cite{MukaiUmemura}.

\begin{proof}
In \cite{Prokhorov} the Fano $3$-folds of index one and $g\ge 7$ carrying vector fields are classified. There are two
rigid examples (Mukai-Umemura $3$-fold with $\Aut^0(X)= \mathbb P SL(2,\mathbb C)$ and a $3$-fold with $\Aut^0(X)=(\mathbb C,+)$)
and a one parameter family of examples with $\Aut^0(X)= (\mathbb C^*,\cdot)$. All the  cases  can be obtained from $X_5$, the Fano $3$-fold of index two
and degree $5$, by means of a birational transformation defined by a linear system on $X_5$ of the form $|3 H - 2Y|$ where $Y$ is the closure
of a $(\mathbb C,+)$ or $(\mathbb C^*,\cdot)$-orbit in $X_5$. Thus Lemma \ref{L:X5} implies that the vector fields in $X$ have, exactly as the vector fields
in $X_5$, isolated fixed points.

Theorem \ref{T:rr} implies that any codimension one foliation on $X$ with $K\mathcal F=0$ must be induced by an algebraic group.
It follows that $X$ is the Mukai-Umemura $3$-fold and that $\mathcal F$ is induced by an action of the affine group.
\end{proof}

\begin{remark}
In the main result of \cite{JR}  there is an imprecision. They claim that a general section of $H^0(X,\Omega^1_X(1))$ for a general deformation of the Mukai-Umemura $3$-fold is integrable. This cannot happen since $h^0(X,\Omega^1_X(1))=3$ for any sufficiently small  deformation of the Mukai-Umemura $3$-fold (\cite[Proposition 2.6]{JR}) and therefore the closedness of Frobenius integrability condition  would imply  that every element of $H^0(X,\Omega^1_X(1))\simeq (\mathfrak{sl}_2)^*$ is integrable. Apparently, their mistake is at the proof of their Proposition 2.16. More specifically, at the determination of the integer $a$ from the exact sequence
$0 \to \mathcal O_{\P^1} \to \mathcal O_{\P^1}(2) \oplus \mathcal O_{\P^1} \oplus \mathcal O_{\P^1}(-1) \to
\mathcal O_{\P^1}(-a+1) \oplus \tau \to 0$, where $\tau$ is a torsion sheaf.
\end{remark}

\section{Holomorphic Poisson structures}\label{S:Poisson} A (non-trivial) holomorphic Poisson structure on projective manifold $X$ is an element of
$[\Pi] \in \P H^0(X, \bigwedge^2 TX)$ such that $[ \Pi,\Pi]=0$, where $[\cdot,\cdot]$ is the Schouten bracket, see \cite{Polishchuk}. In dimension three,
a Poisson structure is equivalent to a pair $(\F,D)$ where $\F$ is a codimension one foliation with $K\F = \mathcal O_X(-D)$
and an effective divisor $D$. Our classifications of irreducible components of the space of foliations with $K\F$ very negative (Proposition \ref{P:verynegative})
and with $K \F =0$  on Fano $3$-folds with  rank one Picard group implies at once a description of the irreducible components of the space of Poisson structures
\[
\mathrm{Poisson}(X)  = \left\{ \Pi \in \P H^0(X, \bigwedge^2 TX) \, \Big| \,  [\Pi,\Pi]=0 \right\} \,
\]
on these manifolds.

\begin{thm}\label{T:Poisson}
If $X$ is a Fano $3$-fold with rank one Picard group then $\mathrm{Poisson}(X)$ has
$6$ irreducible components when $X=\P^3$;
$3$ irreducible components when $X=Q^3$;
 $2$ irreducible components when $X=X_5$;
  $1$ irreducible component when $X$ has index two and is distinct from $X_5$;
 $1$ irreducible component when $X$ is the Mukai-Umemura $3$-fold;
and is empty when $X$ has index one and is not the Mukai-Umemura $3$-fold.
\end{thm}
\begin{proof}
It suffices to verify that every Poisson structure with divisorial zero locus can be deformed to a Poisson structure with codimension two singular set.
A foliation of degree $0$ on $\mathbb P^3$ can be defined by a closed logarithmic $1$-form without divisorial zeros and with poles on two hyperplanes, i.e.
$\Rat(1,1)=\Log(1,1)$. If we multiply a $1$-form with coefficients in $\mathcal O_{\mathbb P^3}(2)$ defining a degree zero foliation by a section of $\mathcal O_{\mathbb P^3}(2)$,
we can deform it to a $1$-form with corresponding belonging  to $\Log(1,1,2)$, as shown below
\[
\lim_{\varepsilon \to 0} Q H_1 H_2 \left( (1+ 2\varepsilon)  \frac{dH_1}{H_1} - \frac{dH_2}{H_2}  - \varepsilon \frac{dQ}{Q}\right) = Q H_1 H_2 \left(  \frac{dH_1}{H_1} - \frac{dH_2}{H_2} \right).
\]
Similar arguments show that we can deform the product of a linear form by a elements of $\Rat(1,2)$ and $\Log(1,1,1)$ (the irreducible components of the space of
degree $1$ foliations on $\mathbb P^3$) to foliations in $\Log(1,1,2)$ and $\Log(1,1,1,1)$ respectively; and that we can deform the product of an integrable   section of $\Omega^1_{Q^3}(2)$ (thus an element of $\Rat(1,1)$ in $Q^3$) by a section of $\mathcal O_{Q^3}(1)$ to an element of $\Log(1,1,1)$.
\end{proof}

\bibliographystyle{amsplain}

\end{document}